\documentclass[a4paper,11pt,oneside]{amsart}

\usepackage{amsmath, amsthm, amscd,  amsfonts,}
\usepackage{graphicx,float}
\usepackage{lineno,hyperref}
\usepackage{mathrsfs}
\usepackage{amssymb}
\usepackage{color}
\usepackage{dsfont}
\usepackage{fancybox}
\usepackage{comment}

\usepackage[english]{babel}

\newcommand{\even}{\mathrm{Even}}
\newcommand{\TP}{\mathrm{TP}}

\newcommand{\Ult}{\mathrm{Ult}}

\newcommand*\axiomfont[1]{\textsf{\textup{#1}}}

\newcommand\gch{\axiomfont{GCH}}
\newcommand\sch{\axiomfont{SCH}}
\newcommand\ch{\axiomfont{CH}}
\newcommand{\crit}{\mathrm{crit}}
\newcommand{\cof}{\mathrm{cof}}
\newcommand{\otp}{\mathrm{otp}}
\newcommand{\dom}{\text{dom}\,}
\newcommand{\Add}{\mathrm{Add}}

\newtheorem{theo}{Theorem}[section]

\newtheorem{defi}[theo]{Definition}
\newtheorem{lemma}[theo]{Lemma}
\newtheorem{prop}[theo]{Proposition}
\newtheorem{remark}[theo]{Remark}
\newtheorem{convention}[theo]{Convention}
\newtheorem{claim}[theo]{Claim}
\newtheorem{subclaim}[theo]{Subclaim}

\newtheorem{notation}[theo]{Notation}

\def\s{\subseteq}
\newcommand{\one}{\mathds{1}}
\def\l{\langle}
\def\r{\rangle}

\title[Tree property at first and double successors]{The tree property at first and double successors of singular cardinals with an arbitrary gap}

\begin{document}

\author{Alejandro Poveda}
\address{Departament de Matem\`atiques i Inform\`atica, Universitat de Barcelona. 
Gran Via de les Corts Catalanes, 585,
08007 Barcelona, Catalonia.}
\email{alejandro.poveda@ub.edu}
\subjclass[2000]{Primary: 03Exx. Secondary: 03E50, 03E57.}
\thanks{ This research has been supported by MECD (Spanish Government) Grant no FPU15/00026, MEC project number MTM2017-86777-P and SGR (Catalan Govern\-ment) project number 2017SGR-270.}

\maketitle
\begin{abstract}
Let $\cof(\mu)=\mu$ and $\kappa$  be  a supercompact cardinal with $\mu<\kappa$. Assume that there is an increasing and continuous sequence of cardinals $\langle\kappa_\xi\mid \xi<\mu\rangle$ with $\kappa_0:=\kappa$ and such that, for each $\xi<\mu$, $\kappa_{\xi+1}$ is supercompact. Besides, assume that $\lambda$ is a weakly compact cardinal with $\sup_{\xi<\mu}\kappa_\xi<\lambda$. Let $\Theta\geq\lambda$ be a cardinal with $\cof(\Theta)>\kappa$. Assuming the $\gch_{\geq\kappa}$, we construct a generic extension where $\kappa$ is strong limit, $\cof(\kappa)=\mu$, $2^\kappa= \Theta$ and both $\TP(\kappa^+)$ and $\TP(\kappa^{++})$ hold. Further, in this model there is a very good and a bad scale at $\kappa$. This  generalizes the main results of \cite{SinTree} and \cite{FriHon}.
\end{abstract}

\section{Introduction}\label{Intro}
Infinite trees play a central role in infinite combinatorics. Recall that a $\kappa$-tree is called $\kappa$-Aronszajn if it has no cofinal branches. 
Given a re\-gular cardinal $\kappa$ it is said that the tree property holds at $\kappa$, denoted by $\TP(\kappa)$, if every $\kappa$-tree has a cofinal branch. By classical results of K\"{o}nig and Aronszajn it is well-known that $\TP(\aleph_0)$ holds while $\TP(\aleph_1)$ fails.
In 1972, Mitchell proved that assuming the existence of  a weakly compact cardinal  $\kappa$ there is a generic extension by a forcing $\mathbb{M}(\kappa)$ where $\kappa=\aleph_2$, $2^{\aleph_0}=\aleph_2$ and $\TP(\aleph_2)$ holds. Thereby the consistency of a weakly compact cardinal gives an upper bound for the consistency of $\TP(\aleph_2)$. It is worth mentioning that the failure of the \ch\, in Mitchell's model is necessary, for  otherwise, by virtue of Specker's theorem, there would be a special $\aleph_2$-Aronszajn tree. The converse implication is also true on the basis of a theorem of Silver (see e.g. \cite{Jech}) who proved that if $\TP(\aleph_2)$ holds then $\aleph_2$ is a  weakly compact cardinal in $L$. Combining both theorems, it follows that $\TP(\aleph_2)$ is equiconsistent with the existence of a weakly compact cardinal. In this paper we are interested in the forcing devised by Mitchell in \cite{Mit}, as well as in other similar constructions developed by several authors over the years \cite{Abr} \cite{CumFor} \cite{SinTree} \cite{Ung} \cite{FriHal} \cite{FriHon}.

Intuitively, Mitchell forcing $\mathbb{M}(\kappa)$ can be conceived as the amalgam of two components: the first one intended to blow up the power set of $\aleph_0$ to $\kappa$ (\textit{Cohen component}) and the second one devised to collapse the interval $(\omega_1,\kappa)$ (\textit{Collapsing component}). Combining this with a fine analysis of the quotients of $\mathbb{M}(\kappa)$, Mitchell's theorem follows.

In the light of Mitchell's result it is  natural to ask whether it is consistent to have the tree property at two consecutive cardinals. The first result in this direction was due to Abraham, who  proved in 1983 that from the existence of a supercompact cardinal with a weakly compact cardinal above, it is possible to force $\TP(\aleph_2)$ and $\TP(\aleph_3)$ \cite{Abr}. Prima facie it may seem surprising   that for getting the consistency of $\TP(\aleph_2) + \TP(\aleph_3)$ one needs much stronger hypotheses than those assumed by Mitchell: especially considering that the consistency of $\TP(\aleph_2)+\TP(\aleph_4)$ follows from a straightforward application of Mitchell's ideas to two weakly compact cardinals.
But, as Magidor observed, to get the consistency of the tree property at two consecutive cardinals one needs to trascend the level of $0^\sharp$ (see \cite[Theorem 1.1]{Abr}).

Some years later, and building on Abraham's ideas, Cummings and Foreman designed a forcing that, starting with infinitely many supercompact cardinals, yields a generic extension where the tree property holds at $\aleph_n$, for each $2\leq n<\omega$ \cite{CumFor}.
In that paper the authors combined Mitchell's construction with the Prikry-type forcing technology to get a model where $\TP(\kappa^{++})$ holds, $\kappa$ is a strong limit cardinal with $\cof(\kappa)=\omega$, and the $\sch_\kappa$ fails \cite{CumFor}. 
Building on these ideas, as well as on others from \cite{Ung}, Friedman, Honzik and Stejskalov\'a \cite{FriHon} exhibited an argument to obtain arbitrary values of $2^\kappa$ in  Cummings-Foreman's model. In particular this shows that the tree property at the double successor of a strong limit singular cardinal $\kappa$ is consistent with an arbitrary failure of the $\sch_\kappa$. Building on \cite{FriHon} this result was subsequently generalized in \cite{GolPov} to the setting of uncountable cofinalities. 

A related discussion to that described previously is about the existen\-ce of Aronszajn trees at first successors of strong limit singular cardinals.  This problem is related with the proof of the consistency of the failure of the \sch\, at a singular strong limit cardinal. Recall that if $\kappa$ is a measurable cardinal with $2^\kappa\geq\kappa^{++}$ then Prikry forcing yields a generic extension where $\square^*_\kappa$ holds, hence $\TP(\kappa^+)$ fails, and $\sch_\kappa$ fails.\footnote{The consistency of the former hypotheses is exactly the existence of a measurable cardinal $\kappa$ with $o(\kappa)=\kappa^{++}$ as proved by Gitik and Mitchell \cite{Jech}.}
Thus a natural question that arises is if this is essentially the only  possible way to produce a model where the $\sch_\kappa$ fails. More formally, given a singular strong limit cardinal $\kappa$ with $\cof(\kappa)=\omega$ does $\TP(\kappa^+)$ (and, in particular, $\neg \square^*_\kappa$) imply $\sch_\kappa$? This question was originally posed in 1989 by Woodin and other authors (see e.g. \cite{For}) and remained unanswered for a long time. Possibly the most decided attempt towards settling this question was due to Gitik and Sharon, who proved the consistency of $\neg\sch_\kappa+\neg \square^*_\kappa$ from the existence of a $\kappa^{+\omega}$-supercompact cardinal $\kappa$ \cite{GitSh}. Also in Gitik-Sharon model there is a very good scale at $\kappa$, a PCF object of central relevance in cardinal arithmetic (see \cite{SheCard} for definitions). 
Shortly after, Cummings and Foremann observed that the failure of $\square^*_\kappa$ in  Gitik-Sharon's model was due to the existence of a bad  scale at $\kappa$.

The construction of a model for $\neg \sch_\kappa+\TP(\kappa^+)$ finally came from Neeman \cite{Nee}, who starting with $\omega$-many supercompact cardinals was able to combine the ideas from \cite{GitSh} with the analysis of narrow systems of \cite{MagShe} to give rise the desired result. Following up on Neeman's ideas, Sinapova proved in \cite{SinTree} that the Mitchell-like forcing of \cite{Ung} can be used to yield a generic extension where $\TP(\kappa^+)$ and $\TP(\kappa^{++})$ both hold while  $\sch_\kappa$ fails. In fact, subsequent work of Sinapova and Unger showed that this can be also done  for $\kappa=\aleph_{\omega^2}$ \cite{SinUng}.
\medskip

In this paper we aim to combine Sinapova's arguments from \cite{SinTree} with those developed in \cite{Ung},\cite{FriHon} and \cite{GolPov}, in order to get a generic extension where $\TP(\kappa^+)$ and $\TP(\kappa^{++})$ both hold, $\kappa$ is a singular strong limit cardinal with $\cof(\kappa)=\mu$ and  there is an arbitrary failure of the $\sch_\kappa$.  Further, as a consequence of results of Sinapova \cite{Sin}, in our generic extension there will be a very good scale and a bad scale at $\kappa$. The main result of the paper is the following.

\begin{theo}[Main Theorem]\label{MainTheorem1}
Let $\cof(\mu)=\mu$ and $\kappa$ be a supercompact cardinal, with $\mu<\kappa$. Assume that there is an increasing and continuous sequence of cardinals $\langle\kappa_\xi\mid \,\xi<\mu\rangle$ with $\kappa_0:=\kappa$ and $\kappa_{\xi+1}$ being supercompact, for each $\xi<\mu$. Besides, assume that there is a weakly compact cardinal $\lambda$ with $\sup_{\xi<\mu} \kappa_\xi<\lambda$, and let $\Theta\geq\lambda$ be a cardinal  with $\cof(\Theta)>\kappa$. Assuming that the $\gch_{\geq\kappa}$ holds, there is a generic extension of the universe where the following holds:
\begin{enumerate}
\item $\kappa$ is a strong limit cardinal with $\cof(\kappa)=\mu$.
\item All cardinals and cofinalities ${\geq}\lambda$ are preserved, ${(\sup_{\xi<\mu}\kappa_\xi)^+}^V=\kappa^+$ and $\lambda=\kappa^{++}$. 
\item $2^\kappa= \Theta$, hence $\neg \sch_\kappa$.
\item $\TP(\kappa^+)$ and $\TP(\kappa^{++})$ hold.
\item {There is a very good scale and a bad scale at $\kappa$.}
\end{enumerate}
\end{theo}
For the proof of this result we shall make use of some ideas developed in  \cite{SinTree}, \cite{Sin} and \cite{SinUncCof} for the proof of $\TP(\kappa^{+})$ and (5). For the rest of items we will use some other ideas from \cite{Ung},\cite{FriHon} and \cite{GolPov}. The structure of the paper is as follows: In Section \ref{SinapovaPreliminaries} we will give an overview of Sinapova forcing following \cite{Sin}. In Section 3 we will proof a criterion for genericity  for Sinapova forcing, which extends the classical Mathias' criterion for Prikry forcing \cite{Git}. 
This result will be crucial in Section \ref{SectionTheMainConstruction}, where we will present our main forcing construction $\mathbb{R}$, and also in  Section \ref{TPkappa++Section}, where we will prove $V^{\mathbb{R}}\models \TP(\kappa^{++})$. We end up the paper with  Section \ref{TPkappa+Section} proving $V^{\mathbb{R}}\models \TP(\kappa^{+})$. Any non defined notion/notation is either standard or will be properly referred.





\section{An overview on Sinapova forcing}\label{SinapovaPreliminaries}
In this section we will review a forcing construction due to D. Sinapova. 
Our exposition will  follow Sinapova's dissertation \cite{Sin}. 
Originally, \textit{Sinapova forcing} (or also \textit{Diagonal Supercompact Magidor forcing}) was conceived to generalize Gitik-Sharon's (GS) theorem to uncountable cofinalities \cite{GitSh}. Also,  inspired by the subsequent inquiries of Cummings and Foremann \cite{CumForGitikSharon} on GS-model, Sinapova devised this forcing to obtain a generic extension where the following hold:
\begin{enumerate}
\item There is a strong limit cardinal $\kappa$ of uncountable cofinality,
\item $\sch_\kappa$ fails,
\item There is a very good and a bad scale at $\kappa$.
\end{enumerate}
Hereafter, $\mu$, $\kappa$, $\langle \kappa_\xi\mid \xi<\mu\rangle$, $\lambda$ and $\Theta$  will be as in the statement of Theorem \ref{MainTheorem1}. Besides, we will define $\varepsilon:=\sup_{\xi<\mu}\kappa_\xi$ and $\delta:=\varepsilon^+$. Since we are assuming $\gch_{\geq \kappa}$ in our ground model, modulo a suitable preparation, we may assume that $\gch_{\geq\varepsilon}$ holds, $2^{\kappa_\xi}=\kappa^+_\xi$, for each $\xi<\mu$, and that $\{\kappa\}\cup\langle\kappa_{\xi+1}\mid \xi<\mu\rangle$ are Laver indestructible supercompact cardinals.\footnote{In this section and in the latter sections \ref{SectionTheMainConstruction} and \ref{TPkappa++Section} we will simply use that $\kappa$ is Laver indestructible. The indestructibility of $\langle\kappa_{\xi+1}\mid \xi<\mu\rangle$ will be important in Section \ref{TPkappa+Section} for the proof of Lemma \ref{DaggerSinapova}. } Through this and the latter sections we will rely on the following standard convention:
\begin{convention}\label{ConventionDownArrow}
\rm{If $\mathbb{P}$ is a forcing notion and $p\in\mathbb{P}$, we will denote by $\mathbb{P}\downarrow p$ the set of conditions in $\mathbb{P}$ below $p$.}
\end{convention}

\subsection{Sinapova forcing}
 Let $\mathbb{A}:=\Add(\kappa,\Theta)$, $G\subseteq \mathbb{A}$ generic and \linebreak$\langle f_\eta\mid \eta\in \Theta\rangle$ be an enumeration of the generic functions added by this filter.  During this section our ground model will be $V[G]$. The next series of result can be found in \cite[Chapter 2]{Sin}.

\begin{prop}
There is a $\Theta^+$-supercompact embedding $j:V\rightarrow M$ with $\crit(j)=\kappa$, such that, for each $\eta<\delta$, $j(f_\eta)(\kappa)=\eta$. Also, $\kappa_\xi^{<\kappa}\leq \kappa_\xi^+$, for each $\xi<\mu$ limit.
\end{prop}


\begin{prop}
For all $\xi<\mu$ and all $\mathcal{X}\subseteq\mathcal{P}(\mathcal{P}_\kappa(\kappa_\xi))$,  there is a $\kappa_\xi$-supercompact measure $U_\xi$ on $\mathcal{P}_\kappa(\kappa_\xi)$ such that $\mathcal{X}\in \Ult(V, U_\xi)$. Also, there are functions $\langle F^\xi_\eta\mid \eta<\delta\rangle$, $F^\xi_\eta\colon \kappa\rightarrow\kappa$ such that, for each $\eta<\delta$, $j_{U_\xi}(F^\xi_\eta)(\kappa)=\eta$.
\end{prop}

\begin{prop}\label{MeasuresMitchelOrder}
There is a $\lhd$-sequence of measures $\langle U_\xi\mid \xi<\mu\rangle$ (i.e. $U_\xi\in \Ult(V, U_{\xi'})$, for $\xi<\xi'$) and functions $\langle F^\xi_\eta\mid \xi<\mu, \eta<\delta\rangle$, $F^\xi_\eta\colon \kappa\rightarrow\kappa$  such that, $U_\xi$ is a $\kappa_\xi$-supercompact measure on $\mathcal{P}_\kappa(\kappa_\xi)$,  and for all $\xi<\mu$ and $\eta<\delta$, $j_{U_\xi}(F^\xi_\eta)(\kappa)=\eta$.
\end{prop}


\begin{notation}\label{Notation1}
$ $
\rm{\begin{itemize}

\item For $\xi<\mu$, $x\in\mathcal{P}_\kappa(\kappa_\xi)$ and $\kappa\leq \tau\leq \kappa_\xi$,  $\tau_x:=\otp(\tau\cap x)$.
\item For $\xi<\mu$ and $x,y\in\mathcal{P}_{\kappa}(\kappa_\xi)$, $x\prec y$ iff $x\subseteq y$ and ${\kappa_{\xi}}_x<\kappa_y$.
\end{itemize}} 
\end{notation}
Let $\mathfrak{U}=\langle U_\xi\mid \xi<\mu\rangle$ and $\mathfrak{F}=\langle F^\xi_\eta\mid \xi<\mu,\eta<\delta\rangle$ be witness for Proposition \ref{MeasuresMitchelOrder}. Since $\mathfrak{U}$ is a $\lhd$-chain , for each $\zeta<\xi<\mu$, there is a function $x\mapsto \overline{U}^{\zeta}_{\xi,x}$,  over $\mathcal{P}_{\kappa}(\kappa_\xi)$ representing $U_\zeta$ in the ultrapower by $U_\xi$. 
Moreover, by restricting this function to a $U_\xi$-large set, we may assume that each $\overline{U}^{\zeta}_{\xi,x}$ is  a $\kappa_{\zeta_x}$-supercompact measure on $\mathcal{P}_{\kappa_x}({\kappa_\zeta}_x)$. 

\begin{defi}\label{DefinitioOfXxi}
For $\xi<\mu$, let $X_\xi$ be the $U_\xi$-large set of  $x\in \mathcal{P}_\kappa(\kappa_\xi)$ such that
\begin{itemize}
\item[($\alpha$)] $\kappa_x$ is a $(\kappa_\xi)_x$-supercompact cardinal above $\mu$.
\item[($\beta$)] For each $\zeta\leq\xi$,
$\kappa_{\zeta_x}^{<\kappa_x}\leq {\kappa_{\zeta}}^+_x$. If $\xi$ is limit,  $\sup_{\zeta<\xi}\kappa_{\zeta_x}=\kappa_{\xi_x}$. 
\item[($\gamma$)] $\kappa_x<{\kappa_{\xi_x}}$.\footnote{This means that our choice of the $x$'s is coherent with the fact that $\kappa<\kappa_\xi$.}
\end{itemize}
\end{defi}

Analogously to other Prikry-type forcing, Sinapova forcing is   articu\-lated by two components: the first one (\textit{stem}) is responsible of adding a generic club on $\kappa$, and the second one (\textit{large set part}) 
plays the role of supplying the stem with new extensions. For technical reasons it is standard to require the stems to be $\prec$-increasing sequences. Roughly, this constraint guarantees that these stems are sound promises for a generic club in $\kappa$ and also that two different \textit{local versions}  of the forcing do not interfere between them.

Let $\zeta<\xi$ and $x\in X_\xi$ and let $\pi^{\zeta,x}: \mathcal{P}_{\kappa_x}(\kappa_\zeta\cap x) \rightarrow \mathcal{P}_{\kappa_x}(\kappa_{\zeta_x})$ be the usual projection.  Set $U^\zeta_{\xi,x}:=\{A\subseteq \mathcal{P}_{\kappa_x}(\kappa_\zeta\cap x)\mid \pi^{\zeta,x} [A]\in \overline{U}^\zeta_{\xi,x}\}$. This lifting yields a supercompact measure over  $\mathcal{P}_{\kappa_x}(\kappa_\zeta\cap x)$. 
In \cite[Section 2.2]{Sin} the following coherence properties are proved: 
\begin{prop}[Coherence properties]\label{TheBxisets}
$ $
\begin{itemize} 
\item[$(\xi)$] For each $\rho<\zeta<\xi<\mu$ and for $U_\xi$-many $x$'s, $U^{\rho}_{\xi,x}\lhd U^{\zeta}_{\xi,x}.$
\item[$(\xi')$] For each $\xi<\mu$, $$B_\xi=\{x\in X_\xi\mid \forall\zeta,\eta\in\xi \,(\zeta<\eta\rightarrow \overline{U}^{\zeta}_{\xi,x}=[y\mapsto \overline{U}^{\zeta}_{\eta,y}]_{U^{\eta}_{\xi,x}} )\}\in U_\xi.$$
\item[($\star$)] For $\zeta<\xi$ and $A\in U_\zeta$,
$\forall_{U_\xi} x\,(A\cap \mathcal{P}_{\kappa_x}(x\cap \kappa_\zeta))\in U^\zeta_{\xi,x}.$
\item[($\diamond$)] For each $\zeta<\eta<\xi$, $z\in B_\xi$ and $A\in U^{\zeta}_{\xi,z}$,
$$\forall_{U^{\eta}_{\xi,z}}x\, (A\cap \mathcal{P}_{\kappa_x}(x\cap \kappa_\eta))\in U^\zeta_{\eta,x}.$$
\end{itemize}
Set $\mathfrak{B}=\langle B_\xi\mid \xi<\mu\rangle$.
\end{prop}


\begin{defi}[Sinapova forcing]\label{SinapovaForcing}
Under the above conditions, Sinapova forcing with respect to $(\kappa,\mu, \mathfrak{U},\mathfrak{B})$ is the partial order $\mathbb{S}_{(\kappa,\mu, \mathfrak{U},\mathfrak{B})}$\footnote{Formally this definition depends also of the functions representing the different measures. }
whose conditions are pairs $( g, H)$ for which the follo\-wing hold:
\begin{enumerate}
\item $\dom(g)\in[\mu]^{<\omega}$ and $\dom(H)=\mu\setminus\dom(g)$.
\item For each $\xi\in\dom(g)$, $g(\xi)\in B_\xi$ and $\kappa_{g(\xi)}>\theta^{+\mu+1}$.\footnote{ Here $\theta$ is an inaccessible cardinal witnessing \cite[Lemma 2.7]{Sin}.  This requirement is technical and is necessary for the construction of the bad and the very good scale in the generic extension. } Also, $g$ is $\prec$-increasing.
\item For each $\xi\in\dom(H)$, 
\begin{enumerate}
\item If $\xi>\max(\dom(g))$, $H(\xi)\subseteq B_\xi$ and $H(\xi)\in U_\xi$;
\item If $\xi< \max(\dom(g))$ then, setting $\xi_g:=\min(\dom(g)\setminus\xi+1)$ and $x:=g(\xi_g)$,   $H(\xi)\in {U}^{\xi}_{\xi_g, x}$.
\end{enumerate}
\item For $\xi<\zeta$ with $\xi\in\dom(g)$ and $\zeta\in\dom(H)$, $g(\xi)\prec x$, for all $x\in H(\zeta)$. 
\end{enumerate}
For a condition $p=(g, H)$ we say that $g$ is the stem and $H$ the large set of $p$. For $\eta\in\dom(g^p)$, denote $(g,H)_{\upharpoonright\eta}:=(g\upharpoonright\eta, H\upharpoonright\eta)$ and $(g,H)_{\setminus \eta}:=(g\setminus\eta, H\setminus \eta)$.
\end{defi}
\begin{defi}\label{OrderSinapova}
Let $p,q\in\mathbb{S}$.
\begin{itemize}
\item[(a)] $p\leq q$ iff
\begin{enumerate}
\item $g^p\supseteq g^q$,
\item If $\xi\in\dom(g^p)\setminus\dom(g^q)$ then $g^p(\xi)\in H^q(\xi)$,
\item If $\xi\notin\dom(g^p)$, $H^p(\xi)\subseteq H^q(\xi)$,
\end{enumerate}
\item[(b)] $p\leq^* q$ iff $p\leq q$ and both conditions have the same stem.
\end{itemize}
Let $p,q\in\mathbb{S}$ with $g^p=g^q=g$. Define $p\wedge q$ as the condition $r:=(g, H^p\wedge H^q)$, where $H^p\wedge H^q$ is the function with domain $\dom(H^p)$ such that $\xi\mapsto H^p(\xi)\cap H^q(\xi)$.
\end{defi}


An important feature of $\mathbb{S}$  is that, below any $p\in\mathbb{S}$,  $\mathbb{S}\downarrow p$ can be decomposed  as the product of two Sinapova forcings. This feature is shared with  other Prikry-type forcings, as Magidor or Radin, and is crucial to control the combinatorics of  $V^{\mathbb{S}}_\kappa$.  Let us formulate this in more formal terms. 

Let $(g ,G)\in\mathbb{S}$, $\{\langle\xi,x\rangle\}\s g$ and $\xi<\mu$ be limit. For each $\eta<\xi$, set $\mathcal{V}_\eta:=U^\eta_{\xi,x}$  and $\mathfrak{V}=\langle \overline{U^\eta_{\zeta,x}}\mid \eta<\zeta<\xi\r$. Also, for each $\zeta<\xi$, find a sequence $\mathfrak{C}=\langle C_\eta\mid \eta<\xi\rangle$ of $\mathcal{V}_\eta$-large sets witnessing Proposition~\ref{TheBxisets} with respect to $\mathfrak{V}$. Now let  $S_{\langle \xi, x\rangle}:=\{ (g,G) \mid \exists ( h,H) \in\mathbb{S}\, ( g,G) = (h,H)_{\upharpoonright\xi}, \wedge\, h(\xi)=x\}$, 
and set $\mathbb{S}_{\langle\xi,x\rangle}:=(S_{\langle \xi, x\rangle},\leq_{\langle\xi, x\rangle})$, where $\leq_{\langle \xi, x\rangle}$ is the induced order by $\leq$. One may argue that $\mathbb{S}_{\langle\xi, x\rangle}$ is Sinapova forcing with respect to $\langle \kappa_x, \xi, \mathfrak{V},\mathfrak{C}\rangle$,  $\mathbb{S}_{\langle \kappa_x, \xi, \mathfrak{V},\mathfrak{C}\rangle}$. The following is also  immediate.
\begin{prop}[Factorization]\label{FactorizationSinapova}
Let $( g ,G)\in\mathbb{S}$, $\{\langle\xi,x\rangle\}\s g$ and $\xi<\mu$ be limit. There is $( g,G')\leq^* ( g,G)$ such that the following hold:
\begin{enumerate}
\item The restriction map $\pi$ between  $\mathbb{S}\downarrow (g,G')$ and $ \mathbb{S}_{\langle \xi, x\rangle }\downarrow (\emptyset, G'\upharpoonright\xi) $  defines a projection.
\item $\mathbb{S}\downarrow (g,G')$ is isomorphic to $ \mathbb{S}_{\langle \xi, x\rangle }\downarrow (\emptyset, G'\upharpoonright\xi) \times \mathbb{S}_{\langle \kappa, \mu, \mathfrak{U}\setminus \xi+1, \mathfrak{B}\setminus \xi+1\rangle}\downarrow ( g\setminus \xi+1, G'\setminus \xi+1)$.
\end{enumerate}
\end{prop}

Let $S\subseteq \mathbb{S}$ be a generic filter for  Sinapova forcing. Set $g^*:=\bigcup_{p\in \mathcal{S}} g^p$, $\kappa^*_\xi:=\kappa_{g^*(\xi)}$ and $\vartheta_\xi:={\kappa_\xi}_{g^*(\xi)}$, for each $\xi<\mu$. The following is a summary of the main properties of $\mathbb{S}$ and $V[S]$:

\begin{theo}[Properties of $\mathbb{S}$]\label{PropertiesofSinapovaForcing}
$ $
\begin{enumerate}
\item $\mathbb{S}$ is a $\delta$-Knaster forcing notion.
\item $\mathbb{S}$ has the Prikry property: namely, for each $p\in\mathbb{S}$ and each sentence $\varphi$ in the language of forcing, there is $q\leq^*p$ so that $q$ decides $\varphi$.
\item Let $\rho<\kappa$ and let $\xi$ be a limit ordinal such that $\vartheta_\xi^+\leq \rho<\kappa^*_{\xi+1}$. Then, $\mathcal{P}(\rho)^{V[S]}=\mathcal{P}(\rho)^{V[S\upharpoonright\xi]}$. Further, if  $\rho\leq\kappa^*_0$, $\mathcal{P}(\rho)^{V[S]}=\mathcal{P}(\rho)^V$.
\end{enumerate}
\end{theo}
\begin{prop}\label{SomeBasicPropSinapova}
The following  hold in $V[S]$:
\begin{enumerate}
\item All cardinals and cofinalities ${\geq}\delta$ are preserved.
\item Let $\rho<\kappa$ be a $V$-cardinal such that for some limit $\xi<\mu$ and some $k<\omega$, $\vartheta_\xi^+\leq \rho< \kappa^*_{\xi+k}$. Then $\rho$ is preserved and $\cof(\rho)=\cof^V(\rho)$. In particular, for each $\xi<\mu$, $\kappa^*_\xi$ is preserved and thus $\kappa$ also.
\item  $\kappa$ is a strong limit cardinal with $\cof(\kappa)=\mu$ and $2^\kappa=\Theta$. Hence, the $\sch_\kappa$ fails.
\item If $\rho\in(\kappa, \varepsilon]$ is a $V$-regular cardinal,  $\cof(\rho)=\mu$. Thus, all $V$-cardinals $\rho\in(\kappa, \varepsilon]$ are collapsed to $\kappa$. 
\end{enumerate}
\end{prop}

Another  remarkable property of Sinapova model is the existence  of a bad and a very good scale at $\kappa$. The concept of scale is the cornerstone of Shelah's PCF theory  \cite{SheCard}. For more information about these objects see \cite{SheCard}, \cite{CumForMag} or \cite{AbrMag}. In \cite[Section 2.5]{Sin} it is  showed how to define  in $V[S]$ these scales   by using the sequence $\mathfrak{F}$. 


\begin{theo}[Sinapova]\label{SinapovaTheorem}
In $V[S]$ the following hold true:
\begin{enumerate}
\item $\kappa$ is a strong limit cardinal with $\cof(\kappa)=\mu$ and $\delta=\kappa^+$.
\item  $2^\kappa\geq\Theta$, hence  $\sch_\kappa$ fails.
\item There is a  very good and a bad scale at $\kappa$.
\end{enumerate}
\end{theo}

\section{Geometric criterion for genericity for $\mathbb{S}$:}
Hereafter $\mathbb{S}$ will be a shorthand for $\mathbb{S}_{(\kappa,\mu,\mathfrak{U},\mathfrak{B})}$.
The present section we will devoted to the proof a Mathias-like criterion of genericity for $\mathbb{S}$. Our exposition is inspired on \cite{Fuch}, where a similar characterization for Magidor forcing is proved.
\begin{notation}
\rm{ $ $
\begin{itemize}
\item $[\prod_{\xi<\mu}\mathcal{P}_\kappa(\kappa_\xi)]$ stands for the set of all $\prec$-increasing sequences in $\prod_{\xi<\mu}\mathcal{P}_\kappa(\kappa_\xi)$ (c.f. Notation \ref{Notation1}).
\item For $n<\omega$, $[\prod_{\xi<\mu}\mathcal{P}_\kappa(\kappa_\xi)]^{n}$ denote the set of $\prec$-sequences of length $n$ in $\prod_{\xi<\mu}\mathcal{P}_\kappa(\kappa_\xi)$ . Analogously, $[\prod_{\xi<\mu}\mathcal{P}_\kappa(\kappa_\xi)]^{<\omega}$ denotes the set of finite $\prec$-sequences.
\item For $g\in[\prod_{\xi<\mu} \mathcal{P}_\kappa(\kappa_\xi)]^{<\omega}$ we respectively denote by $\max(g)$ and $\min(g)$  the $\prec$-maximum and $\prec$-minimum value of $g$.
\end{itemize}}
\end{notation}

Let $S$ be a $\mathbb{S}$-generic filter over $V$.  This set $S$ yields a function \linebreak$g^*\in [\prod_{\xi<\mu} B_\xi]$, which we will call the Sinapova sequence induced by $S$. In particular, $V[g^*]\s V[S]$. As in Prikry forcing (see \cite[\S1.1]{Git}) there is a way to recover the generic $S$ from the induced sequence $g^*$. 
\begin{defi}
For each $g^*\in [\prod_{\xi<\mu} B_\xi]$, define
\begin{eqnarray*}
S(g^*):=\{ (h,H)\in \mathbb{S}\mid h\s g^*, \wedge\,\forall \xi\notin \dom(h)\,
\exists (f,F)\in \mathbb{S}\\ (f,F)\leq (h,H)\,\wedge\, \xi\in\dom(f)\,\wedge\, f(\xi)=g^*(\xi)\}.
\end{eqnarray*}
\end{defi}
\begin{prop}\label{RecoveringGenerics}
For each $g^*\in [\prod_{\xi<\mu} B_\xi]$, $S(g^*)$ is a filter on $\mathbb{S}$. Moreover, if $S\s \mathbb{S}$ is a generic filter and $g^*$ is the induced Sinapova sequence, $S(g^*)=S$.
\end{prop}
\begin{proof}
The proof is a routine verification. The only point that it is worth mentioning is the following. Suppose that $g^*$ is the sequence induced by $S$, for some generic filter $S\s \mathbb{S}$. It is easy to check that $S\s S(g^*)$. In particular, by maximality of generic filters, $S=S(g^*)$.
\end{proof}
It follows from the above  that if $S$ is $\mathbb{S}$-generic over $V$ and $g^*$ is the corresponding Sinapova forcing then $V[S]=V[g^*]$. The previous proposition suggests the next concept:
\begin{defi}
Let $V$ be an inner model of $W$ and $\mathbb{S}\in V$. A sequence $g^*\in [\prod_{\xi<\mu} B_\xi]\cap W$ is $\mathbb{S}$-generic over $V$ if $S(g^*)$ is a $\mathbb{S}$-generic filter over $V$.
\end{defi}

\begin{prop}\label{GenericityCriterionFirstImplication}
Let $V$ be an inner model of $W$ and $\mathbb{S}\in V$. If $g^*\in[\prod_{\xi<\mu} B_\xi]\cap W$ is $\mathbb{S}$-generic over $V$ then the following hold:
\begin{enumerate}
\item\label{P1} For each sequence $H\in V\cap \prod_{\xi<\mu} U_\xi$, there is $\xi_H<\mu$ such that for all ordinal $\eta\in (\xi_H, \mu)$, $ g^*(\eta)\in H(\eta)$.
\item\label{P2} For each $\xi<\mu$ limit and each $H\in V\cap \prod_{\theta<\xi} U^\theta_{\xi,g^*(\xi)}$, there is $\xi_H<\xi$ such that for all ordinal $\eta\in(\xi_H,\xi)$, $g^*(\eta)\in H(\eta)$.
\end{enumerate}
\end{prop}
\begin{proof}
We shall just sketch the proof for property (2) as the proof for the (1) is analogous. Let $\xi<\mu$ be a limit ordinal and a function $H\in V\cap \prod_{\theta<\xi} U^\theta_{\xi,x}$. Since $g^*$ is generic, we may let $(g,G)\in S(g^*)$ with $g=\{\l \xi, g^*(\xi)\r\}$. Set 
$D_H:=\{(i,I)\leq (g,G)\mid \exists\theta\in\xi\,\forall\eta\in (\theta,\xi)\, I(\eta)\s H(\eta)\}.$
It is not hard to check that $D_H$ is dense below $(g,G)$, hence $D_H\cap S(g^*)\neq \emptyset$. Let $(i,I)$ be a condition in this set and $\theta_i<\xi$ be a witness for $(i,I)\in D_H$. Setting $\xi_H:= \theta_i$ it is routine to check that, for all $\eta\in (\xi_H, \xi)$, $g^*(\eta)\in H(\eta)$.
\end{proof}
The goal of this section is precisely to prove that the above properties already characterize those sequences which are $\mathbb{S}$-generic over $V$.
\begin{theo}[Criterion for genericity]\label{CriterionGenericity}
Let $V$ be an inner model of $W$ and $\mathbb{S}\in V$. For a sequence $g^*\in[\prod_{\xi<\mu} B_\xi]\cap W$,
$g^*$ is $\mathbb{S}$-generic over $V$ if and only if properties~\eqref{P1} and \eqref{P2} of Proposition~\ref{GenericityCriterionFirstImplication} hold.
\end{theo}
We will tackle the proof of Theorem \ref{CriterionGenericity} in the next three subsections.
\subsection{One step extensions and pruned conditions}

\begin{defi}
For each $s\in[\mu]^{<\omega}$, define:
\begin{itemize}
\item The \emph{left} operator $\ell_s$ is the map $\ell_s\colon \mu \rightarrow \mu\cup\{-1\}$ defined by 
$$\ell_s(\xi):=\begin{cases}
\max(s\cap \xi), & \text{if $s\cap \xi\neq \emptyset$;}\\
-1, & \text{otherwise}.
\end{cases}
$$
\item The \emph{right} operator $r_s$ is the map $r_s\colon \mu \rightarrow \mu+1$ defined by $r_s(\xi):=\min((s\cup\{\mu\})\setminus \xi+1)$.
\end{itemize}
\end{defi}

\begin{defi}[One-step extension]
Let $(g,G)\in\mathbb{S}$, $\xi\in\dom(G)$ and $x\in G(\xi)$. Define $( g,G){}^\curvearrowright\{\langle\xi, x\rangle\}$ as the  pair $(f,F)$, where $f:=g\cup\{\langle\xi, x\rangle\}$ and $F$ is the function with $\dom(F)=\dom(G)\setminus\{\xi\}$ defined as 
$$F(\eta):=\begin{cases}
G(\eta)\cap \mathcal{P}_{\kappa_x}(\kappa_\eta\cap x), & \text{if $r_{\dom(f)}(\eta)=\xi$};\\
\{y\in G(\eta)\mid x\prec y\}, & \text{if $\ell_{\dom(f)}(\eta)=\xi$};\\
G(\eta), & \text{otherwise}.
\end{cases}
$$
For $1\leq n<\omega$ and a function $f\in [\prod_{\xi\in s} B_\xi]$ with $s\in[\dom(G)]^{n}$, $(g,G){}^\curvearrowright f$ is defined by recursion as $((g,G){}^\curvearrowright f\upharpoonright n-1){}^\curvearrowright\{\langle s_{n-1}, f(s_{n-1})\rangle\}$.\footnote{By convention, $(g,G){}^\curvearrowright \emptyset:=(g,G)$.}
\end{defi}
\begin{remark}
\rm{Observe that not for all functions $f\in [\prod_{\xi\in s} G(\xi)]$ the pair $(g,G){}^\curvearrowright f$ yields a condition in $\mathbb{S}$: it may be the case that, for some $\langle \xi, f(\xi)\rangle\in f$,  $G(\eta)\cap \mathcal{P}_{\kappa_{f(\xi)}}(\kappa_\eta\cap f(\xi))\notin U^\eta_{\xi, f(\xi)}$, for $r_{\dom(g\cup f)}(\eta)=\xi$. }
\end{remark}
\begin{prop}\label{OneStepProp}
Let $(g,G)\in\mathbb{S}$ and $\xi\in\dom(G)$.
\begin{enumerate}
\item If there is a condition $(f,F)\leq (g,G)$ with $g\,\cup\,\{\langle\xi, x\rangle\}= f$, then $(g,G){}^\curvearrowright\{\langle\xi, x\rangle\}\in\mathbb{S}$. Moreover, this is the $\leq$-greatest condition witnessing this property.
\item There is $(g,G^{\xi,+})\leq^*(g,G)$ such that for all $x\in G^{\xi,+}$, $$(g,G){}^\curvearrowright\{\langle\xi, x\rangle\}\in\mathbb{S}.$$
\end{enumerate}
\end{prop}
\begin{proof}
For (1), observe that it is enough with guaranteeing that $G(\eta)\cap \mathcal{P}_{\kappa_x}(\kappa_\eta\cap x)\in U^\eta_{\xi,x}$, for $\eta<\xi$. Notice that this outright follows from $(f,F)\leq (g,G)$. For (2) we argue as follows. For $\eta\in\dom(G)\setminus\{\xi\}$, set $G^{\xi,+}(\eta):=G(\eta)$. Now let $\nu:=r_{\dom(g)}(\xi)$ and $\sigma:=\ell_{\dom(g)}(\xi)$. Without loss of generality assume that $\nu<\mu$, as otherwise the argument is similar.  By using $(\diamond)$ of Proposition~\ref{TheBxisets} it follows that for each $\rho\in (\sigma, \xi)$, there is $A_\rho\in U^\xi_{\nu, g(\nu)}$ such that for each $x\in A_\rho$, $G(\rho)\cap \mathcal{P}_{\kappa_x}(\kappa_\rho\cap x)\in U^\rho_{\xi,x}$. Set $G^{\xi, +}(\xi):=G(\xi)\cap \bigcap_{\rho\in (\sigma,\xi)}A_\rho$. It is routine to check that $(g,G^{\xi,+})$ is as desired.
\end{proof}
One can appeal recursively to Proposition~\ref{OneStepProp}~(1) to obtain the analogous result for functions $f\in[\prod_{\xi\in s} B_\xi]$, $s\in[\dom(G)]^{<\omega}$.  The next concept will be useful in future arguments.
\begin{defi}
A condition $(g,G)\in\mathbb{S}$ is said to be pruned if for all $s\in[\dom(G)]^{<\omega}$ and all $f\in[\prod_{\xi\in s} G(\xi)]$, $(g,G){}^\curvearrowright f\in \mathbb{S}$. 
\end{defi}
\begin{prop}\label{PrunedExtension}
A condition $(g,G)$ is pruned iff for each $\langle \xi, x\rangle\in G$, 
 $(g,G){}^\curvearrowright \{\langle\xi,x\rangle\}\in\mathbb{S}.$
\end{prop}
\begin{proof}
The first implication is obvious. For the converse let us argue, by induction over $n\geq 1$, that for each $s\in[\dom(G)]^n$ and $f\in[\prod_{\xi\in s} G(\xi)]$, $(g,G){}^\curvearrowright f\in\mathbb{S}$. For $n=1$ this follows from our hypothesis. Also, the inductive step follows by combining the recursive definition of $(g,G){}^\curvearrowright f$, the induction hypothesis and our assumption.
\end{proof}
Arguing similarly to Proposition~\ref{OneStepProp} one can prove the next strengthening of clause (2).
\begin{prop}\label{ExistsPrunedExtension}
Let $(g,G)\in\mathbb{S}$. There is a condition $(g,G^*)\in\mathbb{S}$ $\leq^*$-below $(g,G)$ which is pruned.
\end{prop}

\subsection{The Strong Prikry Property for $\mathbb{S}$}
In this section we will prove that the usual strengthening of the Prikry property known as Strong Prikry property holds for $\mathbb{S}$. For the sake of completeness we formulate this principle in the particular context of Sinapova forcing. 
\begin{notation}
\rm{For $(g,G)\in\mathbb{S}$ and $s\in[\dom(G)]^{<\omega}$,  set $S^{(g,G)}_{s}:=\{(i,I)\leq (g,G)\mid \dom(i)=\dom(g)\cup s\}$.} Let $\mathbb{S}^{(g,G)}_s$ be $S^{(g,G)}_s$ endowed with the induced order. Define $\mathbb{S}^{(g,G)}_{\supseteq s}$ analogously.
\end{notation}
\begin{defi}[Strong Prikry Property]\label{SPP}
We will say that $\mathbb{S}$ has the Strong Prikry Property (SPP, for short) if the following property  holds: For each condition $(g,G)\in\mathbb{S}$ and  each dense open set $D\s \mathbb{S}$,  there is  $(g,G^*)\leq^* (g,G)$ and $s\in[\dom(G)]^{<\omega}$ such that $\mathbb{S}^{(g,G^*)}_{\supseteq s}\s D$.
\end{defi}
\begin{lemma}\label{AlmostSPP}
Let $(g,G)\in\mathbb{S}$, $D\s \mathbb{S}$ be dense open and $s\in [\dom(G)]^{<\omega}$. There is a condition $(g,G_s)\leq^* (g,G)$ be such that
$$(*_s)\;\; \mathbb{S}^{(g,G_s)}_s\cap D\neq \emptyset\;\Longrightarrow\; \mathbb{S}^{(g,G_s)}_{\supseteq s}\s D.$$
\end{lemma}
\begin{proof}
We argue by induction over $n=|s|$. If $n=0$, then we ask whether there is $(g,\tilde{G})\leq^*(g,G)$ witnessing $(*_\emptyset)$.  If the answer to our query is affirmative then we let $G_\emptyset$  be such $\tilde{G}$. Otherwise, set $G_\emptyset:=G$.  It is easy to check that $(g,G_\emptyset)$ is as desired.

Now assume that for $(h,H)\in\mathbb{S}$ and each $t\in[\dom(G)]^n$, there is $(h,G_t)\leq^*(h,H)$ witnessing $(*_t)$. Let $s$ be with $|s|=n+1$. Also, say with $\delta:=\min(s)$. Set $t:=s\setminus\{\delta\}$ and $\xi:=r_{\dom(g)}(\delta)$. For each $y\in G(\delta)$, let $(g_y,G_y):=(g,G){}^\curvearrowright\{\langle \delta, y\rangle\}$ and $(g_y, G_{y, t})\leq^*(g_y,G_y)$ witnessing $(*_t)$. Now look at the set of $y\in G(\delta)$ for which the property $(*_t)$ is non-trivial. Namely, set $X:=\{y\in G(\delta)\mid \mathbb{S}^{(g_y,G_{y,t})}_t\cap D\neq\emptyset\}$. If $X\in U^\delta_{\xi,g(\xi)}$, set $Y:=X$ and, otherwise, let $Y$ to be the complement. Let $(g,G_s)\leq^* (g,G)$ be the diagonalization of $\{(g_y,G_{y,t})\mid y\in Y\}$ (see \cite[Proposition~2.12]{Sin}).

\begin{claim}
$(g,G_s)\leq^* (g,G)$ and witnesses $(*_s)$.
\end{claim}
\begin{proof}[Proof of claim]
The first property is obvious so we are left with verifying that $(*_s)$ holds. Without loss of generality, assume that $\mathbb{S}^{(g,G_s)}_{s}\cap D\neq\emptyset$.
Let $(i,I)\in \mathbb{S}^{(g,G_s)}_{s}\cap D$. By definition of diagonalization, $(i,I)\leq (g_y,G_{y,t})$, where $y=i(\delta)\in Y$. Hence, $(i,I)\in \mathbb{S}^{(g_y, G_{y,t})}_t\cap D$, and thus $y\in X\cap Y$. This shows that $Y=X$. 

Now let $(f,F)\in\mathbb{S}^{(g,G_s)}_{\supseteq s}$. Again, by the definition of diagonalization, $(f,F)\in \mathbb{S}^{(g_y,G_{y,t})}_{\supseteq t}$, for $y=f(\delta)\in Y$. Since $X=Y$,  $\mathbb{S}_t^{(g_y,G_{y,t})}\cap D\neq \emptyset$, hence, by $(*_t)$, $\mathbb{S}^{(g_y,G_{y,t})}_{\supseteq t}\s D$, and thus $(f,F)\in D$. Altogether, $\mathbb{S}^{(g,G_s)}_{\supseteq s}\s D$, which yields $(*_s)$.
\end{proof}
\end{proof}

\begin{lemma}\label{SPP}
Let $(g,G)\in\mathbb{S}$ and $D\s \mathbb{S}$ be dense open. There is a condition $(g,G)^*\leq^* (g,G)$ such that
$$(*)\;\;\forall s\in[\dom(G)]^{<\omega}\,( \mathbb{S}^{(g,G)^*}_s\cap D\neq \emptyset\;\Longrightarrow\; \mathbb{S}^{(g,G)^*}_{\supseteq s}\s D).$$
In particular, $\mathbb{S}$ has the SPP.
\end{lemma}
\begin{proof}
For each $s\in[\dom(G)]^{<\omega}$, let $(g,G_s)\leq (g,G)$ be given by Lemma~\ref{AlmostSPP}. For each $\xi\in \dom(G)$, set $G^*(\xi):=\bigcap\{G_s(\xi)\mid \xi\in s\}$. Observe that $(g,G^*)\in\mathbb{S}$ by Definition~\ref{DefinitioOfXxi}($\alpha$) and $\mu^{<\aleph_0}=\mu$. Evidently, $(g,G)^*:=(g,G^*)$ satisfies $(*)$. For the last clause, since $D$ is dense, there is $s$ with $\mathbb{S}^{(g,G)^*}_s\cap D\neq \emptyset$, so that $S^{(g,G)^*}_{\supseteq s}\s D$.
\end{proof}
One can be a bit more ambitious and require that $(g,G)^*$ and $(g,G)$ would be equal up to some $\xi\in\dom(g)$. More formally, $(g,G)^*_{\upharpoonright\xi+1}=(g,G)_{\upharpoonright\xi+1.}$ This more general result follows by combining Lemma~\ref{SPP} with  the following result: 
\begin{lemma}[Diagonalization]
Let $(g,G)\in\mathbb{S}$, $\xi\in\dom(G)$ and $\eta\in\dom(g)\cap \xi$. Assume that $A\in U_\xi$ and $\mathcal{A}=\langle (g_x, G_x)\mid x\in A\rangle$ is a family of conditions below $(g,G)$ with $g_x:=g\cup \{\langle\xi, x\rangle\}$ and $(g_x,G_x)_{\upharpoonright\eta+1}=(g,G)_{\upharpoonright\eta+1}$. Then, there is $(g,G^*)\leq (g,G)$ such that $(g,G^*)_{\upharpoonright\eta+1}=(g,G)_{\upharpoonright\eta+1}$ which diagonalizes the family $\mathcal{A}$.
\end{lemma}
We omit the proof of the above as it is identical to the proof of \cite[Proposition 2.12]{Sin}. Bearing this in mind, one can use Lemma~\ref{SPP} to prove the following: 
\begin{lemma}\label{SPPimproved}
Let $(g,G)\in\mathbb{S}$, $D\s\mathbb{S}$ be dense open and $\eta\in\dom(g)$. There is $(g,G)^{*,\eta}\leq (g,G)$ such that if $(i,I)\leq (g,G)^{*,\eta}$ is in $D$ then, for each $(j,J)\leq (i,I)_{\upharpoonright\eta+1}{}^\smallfrown (g, G)^{*,\eta}_{\setminus \eta+1}$, $(j,J)\in D$.
\end{lemma}
The proof runs in parallel to \cite[Corollary 2.14]{Sin}: here instead of appealing to \cite[Proposition 2.13]{Sin} one invokes Lemma~\ref{SPP}.

\subsection{The proof of the criterion}
We are now in conditions to complete the proof of Theorem~\ref{CriterionGenericity}. Recall that we are left with showing that if $g^*\in[\prod_{\xi<\mu} B_\xi]\cap W$ witnesses properties \eqref{P1} and \eqref{P2} of Proposition~\ref{GenericityCriterionFirstImplication} then $g^*$ is $\mathbb{S}$-generic over $V$.

\begin{proof}[Proof of Theorem~\ref{CriterionGenericity}]
Towards a contradiction, assume that the implication was false. Let $\kappa$ be the first cardinal for which we can define a Sinapova forcing $\mathbb{S}:=\mathbb{S}_{(\kappa,\mu, \mathfrak{U}, \mathfrak{B})}$ and for which there is some $g^*\in[\prod_{\xi<\mu} \mathfrak{B}(\xi)]$ satisfying \eqref{P1} and \eqref{P2} but not being generic. 

Henceforth $D\s \mathbb{S}$ will be an arbitrary but fixed dense open set. We aim to prove that $D\cap S(g^*)\neq \emptyset$. We will be arguing in a similar fashion to \cite[Theorem 1.12]{Git}.

Set $\mathrm{St}:=\{ g\in [\prod_{\xi<\mu} B_\xi]^{<\omega}\mid \exists G\,(g,G)\in\mathbb{S}\}$. For each $g\in \mathrm{St}$, set $(g,G_g):=\one{}^\curvearrowright g$ and\footnote{Since $g\in \mathrm{St}$ observe that Proposition~\ref{OneStepProp} and the subsequent comments guarantee that $\one{}^\curvearrowright g\in \mathbb{S}$.} 
$$(g,\tilde{G}_g):=\begin{cases}
(g,G_g)^*, & \text{if $g=\emptyset$};\\
(g,G_g)^{*, \xi} & \text{if $\max(\dom(g))=\xi$},
\end{cases}
$$
where $(g,G_g)^*$ and $(g,G_g)^{*, \xi}$ are the conditions given by Lemma~\ref{SPP} and Lemma~\ref{SPPimproved}, respectively.

For each $\xi<\mu$ and $x\in\mathcal{P}_\kappa(\kappa_\xi)$ set $\mathrm{St}_{\xi, x}:=\{g\in\mathrm{St}\mid \dom(g)\s\xi,\, \max(g)\prec x\}$. Observe that $|\mathrm{St}_{\xi,x}|\leq |\mathcal{P}_{\kappa_x}(x)|<\kappa$. Thus, 
 $G^*(\xi):=\bigtriangleup_{x\in\mathcal{P}_\kappa(\kappa_\xi)}\left(\bigcap_{g\in\mathrm{St}_{\xi,x}} \tilde{G}_g(\xi)\right)\in U_\xi$. This process yields a function $G^*\in V\cap \prod_{\xi<\mu}U_\xi$. Set $s:=(\emptyset, G^*)$. 
 Appea\-ling to property \eqref{P1} we find $\xi^*<\mu$ limit such that $g^*(\eta)\in G^*(\eta)$, for each $\eta\in(\xi^*,\mu)$.  Set $g^*_{-}:=g^*\upharpoonright\xi^*$, $\mathcal{V}_\eta:=\overline{U}^\eta_{\xi^*,g^*(\xi^*)}$ and $C_\eta:=\overline{B_\eta\cap \mathcal{P}_{\kappa_{g^*(\xi^*)}}(\kappa_\eta\cap g^*(\xi^*))}$, for each $\eta<\xi^*$. Set $\mathfrak{V}:=\langle \mathcal{V}_\eta\mid \eta<\xi^*\rangle$, $\mathfrak{C}:=\langle C_\eta\mid \eta<\xi^*\rangle$ and $\mathbb{S}_{(\kappa_{g^*(\xi^*)}, \xi^*,\mathfrak{V}, \mathfrak{C})}$ be the corresponding Sinapova forcing. Clearly, $g^*_{-}$ witnesses \eqref{P1} and \eqref{P2}, and  $\kappa_{g^*(\xi^*)}<\kappa$, hence $S(g^*_{-})$ is a generic filter for $\mathbb{S}_{(\kappa_{g^*(\xi^*)}, \xi^*,\mathfrak{V}, \mathfrak{C})}$. Let $p^*_{-}:=(\emptyset, \overline{I\upharpoonright\xi^*})\in S(g^*_{-})$. Define $p^*:=(\{\langle \xi^*, g^*(\xi^*)\rangle, H^*)$, where $\dom(H^*):=\mu\setminus\{\xi^*\}$ and
$$H^*(\eta):=\begin{cases}
I(\eta), & \text{if $\eta<\xi^*$},\\
\{x\in G^*(\eta)\mid g^*(\xi^*)\prec x\}, & \text{if $\xi^*<\eta$},
\end{cases}
$$
where $I(\eta)$ denotes the lifting of $\overline{I}(\eta)$ to $\mathcal{P}_{\kappa_{g^*(\xi^*)}}(\kappa_\eta\cap g^*(\xi^*))$. Clearly, $p^*\in\mathbb{S}$. Moreover, by appealing to Proposition \ref{ExistsPrunedExtension}, we may assume that $p^*$ is pruned.
By a very similar argument to Proposition \ref{FactorizationSinapova} (1), there is a projection between $\mathbb{S}\downarrow p^*$ and $\mathbb{S}_{(\kappa_{g^*(\xi^*)}, \xi^*,\mathfrak{V}, \mathfrak{C})}\downarrow p^*_{-}$. Let $\pi$ be such projection and set $D_{p^*}:=D\cap \mathbb{S}\downarrow p^*$. Clearly, $\pi[D_{p^*}]$ is dense and open in $\mathbb{S}_{(\kappa_{g^*(\xi^*)}, \xi^*,\mathfrak{V}, \mathfrak{C})}\downarrow p^*_{-}$. Since $p^*_{-}\in S(g^*_{-})$, it follows that $S(g^*_{-})\cap \pi[D_{p^*}]\neq \emptyset$. Let $(f,F)\in D_{p^*}$ be such that $\pi(f,F)\in S(g^*_{-})\cap \pi[D_{p^*}]$.
 
\begin{claim}
$(f,F)\leq (g,\tilde{G}_g)$, where $g:=f\upharpoonright\xi^*+1$. 
\end{claim}
\begin{proof}[Proof of claim]
Clearly, $g\s f$. 

$\blacktriangleright$ Let $\xi\in\dom(f)\setminus\dom(g)$. Then $\xi^*<\xi$, so that, since $(f,F)\leq p^*$, $F(\xi)\s H^*(\xi)$. By definition of diagonal intersection, and since $\max(g)=g^*(\xi^*)$, $H^*(\xi)\s \tilde{G}_g(\xi)$.

$\blacktriangleright$ Let $\xi\in\dom(\tilde{G}_g)$. If $\xi^*<\xi$ then one may argue as before that $F(\xi)\s \tilde{G}_g(\xi)$. Thus, assume $\xi<\xi^*$. Since $(g,\tilde{G}_g)_{\upharpoonright \xi+1}= (g,G_g)_{\upharpoonright \xi+1}=(\one{}^\curvearrowright g)_{\upharpoonright\xi+1}$, we have $\tilde{G}_g(\xi)=G_g(\xi)=\mathcal{P}_{\kappa_{g(\eta)}}(\kappa_\xi\cap g(\eta))$, where $\eta:=r_{\dom(g)}(\xi)$. Since $f\upharpoonright \xi^*+1=g\upharpoonright\xi^*+1$, clearly $F(\xi)\in U^\xi_{\eta, g(\eta)}$ and thus $F(\xi)\s \tilde{G}_g(\xi)$.
\end{proof}
Now let $(f^*,F^*)$ be defined as $$(f,F)_{\upharpoonright\xi^*+1}{}^\smallfrown (p^*{}^\curvearrowright g^*\upharpoonright(\dom(f)\setminus \xi^*+1))_{\setminus\xi^*+1}$$
This gives a condition in $\mathbb{S}$, because $p^*$ was pruned and $g^*(\xi)\prec g^*(\eta)\in G^*(\eta)$, for $\eta\in(\xi^*,\mu)$. Observe that $(f^*,F^*)$ is also pruned.
\begin{claim}
$(f^*,F^*)\in D\cap S(g^*)$.
\end{claim}
\begin{proof}[Proof of claim]
By combining the definition of $(g,\tilde{G}_g)$, the above claim and the fact that $(f,F)\in D$, it follows that $(f^*,F^*)\in D$. The verification that $(f^*,F^*)\in S(g^*)$ is mere routine.
\end{proof}
From the above arguments we infer that $D\cap S(g^*)\neq\emptyset$ hence, $g^*$ is $\mathbb{S}$-generic over $V$. This produces a contradiction with our initial assumption on $\kappa$ and $\mathbb{S}$.
\end{proof}

For future reference we also include the proof of a general version of the classical R\"owbottom lemma \cite[Theorem 7.17]{Kan}.
\begin{defi}
Let 
$g\in [\prod_{\xi\in\dom(g)}B_\xi]$ and $s\in[\mu\setminus \dom(g)]^{<\omega}$. A sequence $\langle H_\theta\mid \theta\in s\rangle$, is \emph{amenable} to $\langle g,s\rangle$ if for each $\theta\in s$, if $\eta:=r_{\dom(g)}(\theta)<\mu$, then $H_\theta\in U^\theta_{\eta, g(\eta)}$ and, otherwise, $H_\theta\in U_\theta$.

A sequence $\langle H_\theta\mid \theta\in \mu\setminus \dom(g)\rangle$ is said to be \emph{amenable} to $\langle g\rangle$ if, for each $s\in[\mu\setminus\dom(g)]^{<\omega}$, $\langle H_\theta\mid \theta\in s\rangle$ is amenable to $\langle g,s\rangle$.
\end{defi}

\begin{lemma}[Generalized R\"{o}wbottom's lemma]\label{DiagonalSupercompactRowottom}
Let
$g$ be a sequence in $[\prod_{\xi\in\dom(g)}B_\xi]$ and $\langle H_\theta\mid \theta\in \mu\setminus \dom(g)\rangle$ be amenable to $\langle g\rangle$.

For each function $c:[\prod_{\theta\in \mu\setminus \dom(g)} H_\theta]^{<\omega}\rightarrow \vartheta$  with $\vartheta\leq \mu$, there is $\langle H^*_\theta\mid \theta\in \mu\setminus \dom(g)\rangle$ amenable to $\langle g\rangle$ such that the following hold:
\begin{enumerate}
\item for each $\theta\in\mu\setminus\dom(g)$, $H^*_\theta\subseteq H_\theta$;
\item $\langle H^*_\theta\mid \theta\in \mu\setminus \dom(g)\rangle$ is  homogeneous for $c$: namely, for each $n<\omega$ and each $s\in[\mu\setminus \dom(g)]^n$,  the function $c\upharpoonright[\prod_{\theta\in s} H^*_\theta]$ is constant.
\end{enumerate}
\end{lemma}
\begin{proof}
Arguing by induction over $n<\omega$, we will prove that for each function $\bar{c}:[\prod_{\theta\in \mu\setminus \dom(g)} H_\theta]^{n}\rightarrow \vartheta$ and  $s\in [\mu\setminus \dom(g)]^n$, there is a sequence  $\mathcal{H}^s=\langle H^s_\theta\mid\theta\in s\rangle$ which is amenable to $\langle g,s\rangle$ and such that $c\upharpoonright [\prod_{\theta\in s}H^s_\theta]$ is constant. If $n=1$ the claim  follows by appealing to the $\mu^+$-completedness of all the measures involved (see Definition~\ref{DefinitioOfXxi}($\alpha$)). 
Thus,
we shall assume that the result holds for each $1\leq m\leq n$ and will infer from this that it holds for $n+1$.

Fix $\bar{c}:[\prod_{\theta\in \mu\setminus \dom(g)} H_\theta ]^{n+1}\rightarrow \vartheta$ be a function and let $s\in [\mu\setminus \dom(g)]^{n+1}$. Set $\max(s)=\eta_s$. Say, $\xi_s:=r_{\dom(g)}(\eta_s)$ and assume, for instance, that $\xi_s<\mu$. Thus, $H_{\eta_s}\in U^{\eta_s}_{\xi_{\eta_s}, g(\xi_{\eta_s})}$. For each $g\in[\prod_{\theta\in s\cap \eta_s} H_\theta]$, let  $c_g: H_{\eta_s}\rightarrow \vartheta$ be the function defined  by $x\mapsto \bar{c}(g\cup \{\langle \eta_s, x\rangle\})$, provided  $\max(g)\prec x$, or $0$ otherwise. Appealing to the case $n=1$, for each such $g$ we obtain $\langle H^s_g\rangle $ which is amenable to $\langle g,\{\eta_s\}\rangle$ and homogeneous with respect to $c_g$. Pick $\vartheta_g\in \vartheta$ be the constant value of $c_g$ witnessing this. 
Let $H^s_{\eta_s}=\bigtriangleup\{H^s_g:\,g\in[\prod_{\theta\in s\cap \eta_s} H_\theta]\}$, where recall that this diagonal intersection is defined as
$$
\{x\in \mathcal{P}_{\kappa_{g(\xi_s)}}(\kappa_{\eta_s}\cap g(\xi_s)) \mid \, \forall g\in[\prod_{\theta\in s\cap \eta_s} H_\theta]\;(\max(g)\prec x\,\rightarrow\, x\in H^s_g)\}.
$$
By normality of the measure $U^{\eta_s}_{\xi_s, g(\xi_s)}$, $H^s_{\eta_s}\in U^{\eta_s}_{\xi_{\eta_s}, g(\xi_{\eta_s})}$. On the other hand, let  $c^*:[\prod_{\theta\in \mu\setminus \dom(g)} H_\theta ]^{n}\rightarrow \vartheta$ be the function sending each $g$ to $\vartheta_g$, in case $g \in[\prod_{\theta\in s\cap \eta_s}H_\theta  ]$, or $0$ otherwise. By the induction hypothesis 
there is $\mathcal{H}^{s\cap \eta_s}=\langle H^{s\cap \eta_s}_\theta\mid\,\theta\in s\cap \eta_s\rangle$ which is amenable to $\langle g, s\cap \eta_s\rangle$ and $c^\star\upharpoonright[\prod_{\theta\in s\cap \eta_s} H^{s\cap \eta_s}_\theta]$ has constant value $\vartheta^*$. 

We claim that $\mathcal{H}^s=\mathcal{H}^{s\cap \eta_s}\cup \{\langle \eta_s, H^s_{\eta_s}\rangle \}$ witnesses the inductive step relative to the function $\bar{c}$ and the set $s$. It is easy to check that $\mathcal{H}^s$ is amenable to $\langle g, s\rangle$. For homogeneity, 
let $f\in [\prod_{\theta\in s} H^s_\theta]$ and say that $f=g \cup\{\langle\eta_s, x\rangle\}$, where $g\in [\prod_{\theta\in s\cap \eta_s} H^s_\theta]$. Since $x\in {H}^s_{\eta_s}$ and $\max(g)\prec x$, by definition of diagonal intersection, $x\in H^s_{g}$. Thus, $c_g(x)=\vartheta_g=\vartheta^*$. On the other hand, $\bar{c}(f)=c_{g}(x)$, so that $\bar{c}(f)=\vartheta^*$. Since the choice of $s$ was arbitrary, the inductive step follows.

For each $n<\omega$ use the previous argument to obtain a sequence $\langle \mathcal{H}^s\mid s\in[\mu\setminus\dom(g)]^n\rangle$, $\mathcal{H}^s=\langle H^s_\theta\mid \theta\in s\rangle$, such that $\mathcal{H}^s$ is amenable to $\langle g,s\rangle$ and  $c\upharpoonright [\prod_{\theta\in s} H^s_\theta]$ is constant. Define $\langle H^*_\theta\mid \theta\in\mu\setminus \dom(g)\rangle$ as  $H^*_\theta:=\bigcap\{H^s_\theta: s\in[\mu\setminus \dom(g)]^{<\omega}, \theta\in s\}.$
Since all the measures involved are $\mu^+$-complete this process yields a sequence 
$\langle H^*_\theta\mid \theta\in\mu\setminus \dom(g)\rangle$ which is amenable to $\langle g\rangle$. Finally,
it is routine to check that this sequence is  homogeneous  for  $c$.
\end{proof}

\section{The main forcing construction}\label{SectionTheMainConstruction}
The present section will be devoted to introduce the main forcing construction of the paper. This forcing is a variation of the forcings appe\-aring in \cite{Ung} or in \cite{GolPov}, where the Supercompact Prikry/ Magidor forcing is replaced by Sinapova forcing. This new choice will be the responsible of the very good and the bad scale in the generic extension. For enlightening  the argument we will simply give  details for the construction in case $\Theta=\lambda^{+}$.  
The general definition can be easily inferred from our arguments. For more details we refer the reader to \cite[\S4]{FriHon}.

\begin{notation}\label{CohenForcingNotation}
\rm{$ $
\begin{itemize}
\item For each $x\subseteq \lambda^+$,  $\mathbb{A}_x:=(\Add(\kappa,x),\supseteq)$. 
\item For each $y\subseteq x\subseteq\lambda^+
$ and $H\s\mathbb{A}_x$ a generic filter, $H\upharpoonright y$ will denote the generic filter induced by $H$ and the standard projection between $\mathbb{A}_x$ and $\mathbb{A}_y$.
\end{itemize}
 
}
\end{notation}
Let $G\subseteq \mathbb{A}_{\lambda^+}$ generic over $V$. Since $\kappa$ is Laver indestructible there is in $V[G]$ a $\lhd$-increasing sequence  $\mathfrak{U}_{\lambda^+}=\langle U_\xi\mid \xi<\mu\rangle$ of supercompact measures on $\mathcal{P}_\kappa(\kappa_\xi)$, $\xi<\mu$. With $\mathfrak{U}_{\lambda^+}$ we find a sequence  $\mathfrak{B}_{\lambda^+}=\langle B_\xi\mid \xi<\mu\rangle$ witnessing Proposition~\ref{TheBxisets} and later define the corresponding 
Sinapova forcing $\mathbb{S}_{\lambda^+}:=\mathbb{S}_{(\kappa,\mu,\mathfrak{U},\mathfrak{B})}\in V[G]$. 
For each such $\xi$, let $\dot{U}_\xi$ and $\dot{B}_\xi$ be $\mathbb{A}_{\lambda^+}$-nice names for each of such objects. The next result shows that there are many intermediate extensions of $V[G]$ where $(\mathfrak{U}_{\lambda^+}, \mathfrak{B}_{\lambda^+})$ \textit{projects}. For details the reader is referred to \cite[Lemma 3.3]{FriHon} or to \cite[Lemma 3.1]{GolPov} where a similar result is proved.

\begin{lemma}\label{DefinitionOFA}
There is an unbounded set of ordinals $\mathcal{A}\subseteq \lambda^+$, closed under taking limits of ${\geq}\delta$-se\-quences, such that, for each $\alpha\in\mathcal{A}$ and each generic filter $G\subseteq \mathbb{A}_{\lambda^+}$, $\langle ({\dot{U}_\xi})_G\cap V[G\upharpoonright \alpha]\mid \xi<\mu\rangle$, 
$\langle ({\dot{B}_\xi})_G\cap V[G\upharpoonright \alpha]\mid\xi<\mu\rangle$ 
are suitable to define Sinapova forcing in $V[G\upharpoonright \alpha]$.
\end{lemma}

\begin{notation}
\rm{For each $\alpha\in\mathcal{A}$, let $\mathfrak{U}_\alpha$ and $\mathfrak{B}_\alpha$ 
be the sequences witnessing Lemma \ref{DefinitionOFA}. Let $\dot{\mathbb{S}}_\alpha$ be a $\mathbb{A}_\alpha$-name representing the Sinapova forcing $\mathbb{S}_{(\kappa,\mu,\mathfrak{U}_\alpha,\mathfrak{B}_\alpha)}\in V[G\upharpoonright\alpha]$. 
}
\end{notation}
\begin{prop}\label{PropOnProjectionOfGenerics}
Work in $V[G]$. For each $\alpha\in\mathcal{A}$, $\mathbb{S}_{\lambda^+}$ projects onto $\mathbb{S}_\alpha$.
\end{prop}
\begin{proof}
Let $\alpha\in\mathcal{A}$.
For $g^*\in[\prod_{\xi<\mu} B_\xi]$ being a $\mathbb{S}_{\lambda^+}$-generic sequence  set $h^*_\alpha:=\langle g^*(\xi)\cap V[G\upharpoonright\alpha]\mid \xi<\mu\rangle$. Clearly, $h^*_\alpha\in[\prod_{\xi<\mu}\mathfrak{B}_\alpha(\xi)]$. By appealing to Theorem~\ref{CriterionGenericity} we infer that $h^*_\alpha$ is $\mathbb{S}_\alpha$-generic over $V$. In particular, each $\mathbb{S}_{\lambda^+}$-generic filter induces a $\mathbb{S}_\alpha$-generic filter, hence $\mathbb{S}_{\lambda^+}$ projects onto $\mathbb{S}_\alpha$.
\end{proof}

Before presenting our forcing it is convenient to discuss a technical issue that we will have to overcome. If one looks at Mitchell's original proof of $\TP(\aleph_2)$ \cite{Mit} one will immediately realize that both the \textit{Cohen component} and the \textit{collapsing component} need to have the same length. More formally, if we aim to add $\lambda^+$-many subsets to $\kappa$ (i.e. the Cohen part is $\Add(\kappa,\lambda^+)$) then the collapsing component will collapse the interval $(\kappa,\lambda^+)$. Thus, if one  pretends to preserve $\lambda$,  the corresponding Mitchell forcing should exhibit a 
mismatch between both components. To overcome this difficulty we shall proceed as in \cite{FriHon} and \cite{GolPov} defining a system of projections between $\mathbb{A}_{\lambda^+}\ast \dot{\mathbb{S}}_{\lambda^+}$ and a family of intermediate forcings.

Let $\beta_0\in\mathcal{A}\setminus \lambda+1$ and $\pi:\beta_0\rightarrow \mathrm{Even}(\lambda)$ be a bijection\footnote{For an ordinal ${\alpha}$, $\mathrm{Even}(\alpha)$ stands for the set of all even and limit ordinals ${\leq} \alpha$. }.  Hereafter, $\beta_0$ will be fixed. The particular choice of this ordinal is not relevant, we could just have taken any other in $\mathcal{A}\setminus \lambda+1$. Clearly, $\pi$ entails an $\in$-isomorphism between $V^{\mathbb{A}_{\beta_0}}$ and $V^{\mathbb{A}_{\mathrm{Even}(\lambda)}}$. Thus, defining $\dot{\mathfrak{U}}^\pi_{\beta_0}:= \pi(\dot{\mathfrak{U}}_{\beta_0})$, $(\mathfrak{U}^\pi_{\beta_0})_{\pi[G\upharpoonright\beta_0]}=(\dot{\mathfrak{U}}_{\beta_0})_{G\upharpoonright\beta_0}=\mathfrak{U}_{\beta_0}$. Similarly with $\mathfrak{B}_{\beta_0}$. Say that $U^\pi_\xi$ and $B^\pi_\xi$ 
are the components of these sequences. For the ease of notation, let  $H$ be the $\mathbb{A}_{\mathrm{Even}(\lambda)}$-generic filter generated by $\pi[G\upharpoonright\beta_0]$.  The proof of the next result is analogous to Lemma \ref{DefinitionOFA}.
\begin{lemma}\label{DefinitionOFB}
There is an unbounded set of cardinals $\mathcal{B}\subseteq \lambda$ closed under taking limits of  ${\geq}\delta$-se\-quences, such that for each $\alpha\in\mathcal{B}$ and each generic filter $K\subseteq \mathbb{A}_{\mathrm{Even}(\lambda)}$, the sequences $\langle ({\dot{U}^\pi_\xi})_G\cap V[K\upharpoonright \mathrm{Even}(\alpha)]\mid \xi<\mu\rangle$ and
$\langle ({\dot{B}^\pi_\xi})_G\cap V[K\upharpoonright \mathrm{Even}(\alpha)]\mid \xi<\mu\rangle$ 
are suitable to define Sinapova forcing in   $V[K\upharpoonright \mathrm{Even}(\alpha)]$.
\end{lemma}
\begin{notation}
\rm{
For each $\alpha\in\mathcal{B}$, let  $\mathfrak{U}^\pi_\alpha$ and $\mathfrak{B}^\pi_\alpha$ denote the sequences witnessing Lemma \ref{DefinitionOFB}. By convention, $\mathfrak{U}^\pi_\lambda:=\mathfrak{U}_{\beta_0}$ and $\mathfrak{B}^\pi_\lambda:=\mathfrak{B}_{\beta_0}$.
For each $\alpha\in\mathcal{B}\cup \{\lambda\}$,  let $\dot{\mathbb{S}}^\pi_\alpha$ be a $\mathbb{A}_{\mathrm{Even}(\alpha)}$-name for the Sinapova forcing $\mathbb{S}_{(\kappa,\mu,\mathfrak{U}^\pi_\alpha,\mathfrak{B}^\pi_\alpha)}\in V[H\upharpoonright\mathrm{Even}(\alpha)]$. 
}
\end{notation}
The next lemma follows essentially from Proposition~\ref{PropOnProjectionOfGenerics}. For details see \cite[Lemma 3.8]{FriHon}.
\begin{lemma}\label{ProjectionsCohenPart}
Let $\hat{\mathcal{A}}=(\mathcal{A}\cap (\beta_0,\lambda^+))\cup\{\lambda^+\}$. 
\begin{enumerate}
\item For every $\gamma,\tilde{\gamma}\in\hat{\mathcal{A}}$ with $\gamma<\tilde{\gamma}$, there is a projection
$$
\sigma^{\tilde{\gamma}}_\gamma: \mathbb{A}_{\tilde{\gamma}}\ast \dot{\mathbb{S}}_{\tilde{\gamma}}\rightarrow \mathrm{RO}^+(\mathbb{A}_{\gamma}\ast \dot{\mathbb{S}}_\gamma).
$$
\item For every $\gamma\in\hat{\mathcal{A}}$ and $\alpha\in\mathcal{B}$, there is a projection
$$
\sigma^{\gamma}_\alpha: \mathbb{A}_\gamma\ast \dot{\mathbb{S}}_\gamma\rightarrow \mathrm{RO}^+(\mathbb{A}_{\rm{Even}(\alpha)}\ast \dot{\mathbb{S}}^\pi_\alpha).
$$
\item For every $\gamma\in\hat{\mathcal{A}}$ and $\alpha\in\mathcal{B}$, let $\hat{\sigma}^\gamma_\alpha$ be the extension of $\sigma^\gamma_\alpha$ to the Boolean completion of $\mathbb{A}_{\gamma}\ast \dot{\mathbb{S}}_\gamma$
$$
\hat{\sigma}^\gamma_\alpha: \mathrm{RO}^+(\mathbb{A}_\gamma\ast \dot{\mathbb{S}}_\gamma)\rightarrow \mathrm{RO}^+(\mathbb{A}_{ \mathrm{Even}(\alpha)}\ast \dot{\mathbb{S}}^\pi_\alpha).
$$
Then the projections commute with $\sigma^{\lambda^{+}}_\alpha$:
$$
\sigma^{\lambda^{+}}_\alpha=\hat{\sigma}^{\gamma}_\alpha\circ \sigma^{\lambda^{+}}_\gamma.$$
\end{enumerate}
\end{lemma}

\begin{defi}[Main forcing]\label{MainForcingR}
A condition in $\mathbb{R}$ is a triple $(p,\dot{q},r)$ for which all the following hold:
\begin{enumerate}
\item $(p,\dot{q})\in\mathbb{A}_{\lambda^+}\ast\dot{\mathbb{S}}_{\lambda^+}$;
\item $r$ is a partial function with $\dom(r)\in [\mathcal{B}]^{< \delta}$;
\item For every $\gamma\in\dom(r)$, $r(\gamma)$ is a $\mathbb{A}_{\rm{Even}(\gamma)}\ast \dot{\mathbb{S}}^\pi_\gamma$-name such that
$$\one_{\mathbb{A}_{\rm{Even}(\gamma)}\ast \dot{\mathbb{S}}^\pi_\gamma}\Vdash_{\mathbb{A}_{\rm{Even}(\gamma)}\ast \dot{\mathbb{S}}^\pi_\gamma}\text{$``r(\gamma)\in \dot{\Add}(\delta, 1)$''}.$$
\end{enumerate}
For conditions $(p_0,\dot{q}_0,r_0), (p_1,\dot{q}_1, r_1)$ in $\mathbb{R}$ we will write $(p_0,\dot{q}_0,r_0)\leq_\mathbb{R} (p_1,\dot{q}_1, r_1)$ iff $(p_0,\dot{q}_0)\leq_{\mathbb{A}_{\lambda^+}\ast \dot{\mathbb{S}}_{\lambda^+}}(p_1,\dot{q}_1)$, $\dom(r_1)\subseteq \dom(r_0)$ and for each $\gamma\in\dom(r_1)$,
$\sigma^{\lambda^{+}}_\gamma(p_0,q_0)\Vdash_{\mathbb{A}_{\rm{Even}(\gamma)}\ast \dot{\mathbb{S}}^\pi_\gamma}\text{$``r_0(\gamma)\leq r_1(\gamma)$''}.$
\end{defi}
\begin{defi}
$\mathbb{U}$ will denote the pair $(U,\leq)$ where $U:=\{(\one,\dot{\one},r)\mid (\one,\dot{\one},r)\in\mathbb{R}\}$ and $\leq$ is the  order inherited from $\mathbb{R}$. Set $\bar{\mathbb{R}}:=(\mathbb{A}_{\lambda^+}\ast\dot{\mathbb{S}}_{\lambda^+})\times \mathbb{U}$.
\end{defi}
The next result follows from standard arguments.
\begin{prop}$ $\label{projectionU}
\begin{enumerate}
\item $\mathbb{U}$ is $\delta$-directed closed.
\item The function $\rho: \bar{\mathbb{R}}\rightarrow \mathbb{R}$ given by $\langle(p,\dot{q}), (\one,\dot{\one},r)\rangle\mapsto (p,\dot{q},r)$ entails a projection. In particular,
$V^{\mathbb{A}_{\lambda^+}\ast \dot{\mathbb{S}}_{\lambda^+}}\subseteq V^\mathbb{R}\subseteq V^{\bar{\mathbb{R}}}.$
\item $V^{\mathbb{A}_{\lambda^+}\ast \mathbb{S}_{\lambda^+}}$ and $V^\mathbb{R}$ have the same ${<}\delta$-sequences. 
\end{enumerate}
\end{prop}

Let $\bar{R}\subseteq \bar{\mathbb{R}}$ a generic filter whose projection onto $\mathbb{A}_{\lambda^+}$ generates the generic filter $G$. Also, let $R\subseteq \mathbb{R}$ be the generic filter generated by $\rho[\bar{R}]$ and $S\subseteq\mathbb{S}_{\lambda^+}$ be the generic filter over $V[G]$ induced by $\bar{R}$. 

\begin{prop}[Some properties of $\mathbb{R}$]$ $\label{PropertiesOfV[R]}
\begin{enumerate}
\item $\mathbb{R}$ is $\lambda$-Knaster. In particular, all $V$-cardinals ${\geq}\lambda$ are preserved.
\item $\mathbb{R}$ preserves $\kappa$ and $\delta$. Also, it collapses all the $V$-cardinals of $(\kappa,\delta)$ to $\kappa$ and all the $V$-cardinals of $(\delta,\lambda)$ to $\delta$. In particular, $V[R]\models \text{$``\delta=\kappa^+\,\wedge\,\lambda=\kappa^{++}$''}$.
\item $V[R]\models\text{$``2^\kappa= \lambda^{+}=\kappa^{+3}$''}$.
\item $V[R]\models \text{$``\kappa$ is strong limit with $\cof(\kappa)=\mu$''}$.
\item In $V[R]$ there is a bad and a very good scale at $\kappa$. In particular, $\square^*_\kappa$ fails and thus there are no special $\kappa^+$-Aronszajn trees.
\end{enumerate}
\end{prop}

\begin{proof}
$ $
\begin{enumerate}
\item It follows from a similar argument to \cite[Lemma 3.6]{GolPov}.

\item  Let $\theta\in\{\kappa,\delta\}\cup (\kappa,\delta)\cup (\delta, \lambda)$ and let us discuss what happens in each case. If $\theta=\kappa$ it is enough to prove that $\mathbb{A}_{\lambda^+}\ast \dot{\mathbb{S}}_{\lambda^+}$ preserves it, and this follows from a standard argument combining the $\kappa$-closedness of $\mathbb{A}_{\lambda^+}$ with the Prikry property and the $\kappa$-closedness of $\langle \dot{\mathbb{S}}_{\lambda^+}, \leq^*\rangle$. If $\theta=\delta$ the argument is similar but now appealing to Easton's lemma (see e.g. \cite{Kun}). If $\theta\in (\kappa,\delta)$, it is clear that $\mathbb{R}$ collapses $\theta$ because there is a projection between $\mathbb{R}$ and $\mathbb{A}_{\lambda^+}\ast \dot{\mathbb{S}}_{\lambda^+}$, and this last forcing collapses the interval $(\kappa,\delta)$ (cf Proposition \ref{SomeBasicPropSinapova}(4)). Finally, assume that $\theta\in (\delta,\lambda)$ and let $\eta\in\mathcal{B}\cap (\delta,\lambda)$ with $\eta>\theta$. It is easy to see that there is a projection between $\mathbb{R}$ and $\mathrm{RO}^+(\mathbb{A}_{\even(\eta)}\ast \dot{\mathbb{S}}^\pi_\eta)\ast \dot{\mathrm{Add}}(\delta,1)$. By standard arguments this latter iteration collapses the interval $(\delta, \eta]$ and thus $\theta$.

\item[(3)] The first equality follows by counting $\mathbb{R}$-nice names and from the existence of a projection between $\mathbb{R}$ and $\mathbb{A}_{\lambda^+}$. For the latter, use item (2).

\item[(4)] Clearly it suffices to argue that in $V[G\ast \dot{S}]$ the property holds. Nevertheless, observe that this is true by Proposition~\ref{SomeBasicPropSinapova}(3). 

\item[(5)] This follows from the existence of a very good (resp. bad) scale in $V[G\ast \dot{S}]$ (see Theorem \ref{SinapovaTheorem}), $(\kappa^+)^{V[G\ast \dot{S}]}=(\kappa^+)^{V[R]}=\delta$ and the fact that $V[G\ast\dot{S}]$ and $V[R]$ have the same ${<}\delta$-sequences.
\end{enumerate}
\end{proof}

\section{$\TP(\kappa^{++})$ holds}\label{TPkappa++Section}
In the present section we will prove that $V[R]\models \TP(\kappa^{++})$. For enlightening  the presentation, once again, we will simply give details for the proof in case $\gamma=\lambda^+$. A sketch of the main ideas involved in the proof of the more general result can be found in \cite{FriHon} or in \cite{GolPov}.

\medskip

Let us briefly summarize the structure of the argument. First we beging proving that any counterexample for $\TP(\lambda)$ in $V[R]$ lies in an intermediate extension of $\mathbb{R}$. More formally, 
any $\lambda$-Aronszajn tree in $V[R]$ is a $\lambda$-Aronszajn tree in a generic extension given by some truncation of $\mathbb{R}$ (see Proposition~\ref{TruncationsAndAronszajnTree}).  These truncations have the important feature that they are isomorphic to a Mitchell forcing $\mathbb{R}^*$ without  mismatches between the Cohen and the collapsing component.

In latter arguments we shall again consider truncations of $\mathbb{R}^*$, $\mathbb{R}^*\upharpoonright\gamma$, and use the weak compactness of $\lambda$ to prove that any $\lambda$-Aronszajn tree in  $V^{\mathbb{R}^*}$ reflects to a $\gamma$-Aronszajn tree in $V^{\mathbb{R}^*\upharpoonright\gamma}$ (see Lemma \ref{AronszajnWeakComp}). Then, we will be in conditions to use Unger's ideas \cite{Ung} to show that there are no $\gamma$-Aronszajn trees in $V^{\mathbb{R}^*\upharpoonright \gamma}$, and thus that $V[R]\models \TP(\lambda)$. Let $\beta_0\in\mathcal{A}\setminus\lambda+1$ be the ordinal fixed in the previous section.




\begin{defi}[Truncations of $\mathbb{R}$]\label{truncationsR}
Let $\alpha\in\mathcal{A}\cap (\beta_0,\lambda^+)$. 
A condition in $\mathbb{R}\upharpoonright\alpha$ is a triple $(p,\dot{q},r)$ for which all the following hold:
\begin{enumerate}
\item $(p,q)\in \mathbb{A}_\alpha\ast \dot{\mathbb{S}}_\alpha$; 
\item $r$ is a partial function with $\dom(r)\in[\mathcal{B}]^{<\delta}$;
\item For every $\beta\in\dom(r)$, $r(\beta)$ is a $\mathbb{A}_{\rm{Even}(\beta)}\ast\dot{\mathbb{S}}^\pi_\beta$-name such that
$$\one_{\mathbb{A}_{\mathrm{Even}(\beta)}\ast \dot{\mathbb{S}}^\pi_\beta}\Vdash_{\mathbb{A}_{\mathrm{Even}(\beta)}\ast\dot{\mathbb{S}}^\pi_\beta} \text{$``\dot{r}(\beta)\in\dot{\Add}(\delta,1)$''}.$$
\end{enumerate}
For conditions $(p_0,\dot{q}_0,r_0),(p_1,\dot{q}_1,r_1)$ in $\mathbb{R}\upharpoonright \alpha$ we will write $(p_0,\dot{q}_0,r_0)\leq (p_1,\dot{q}_1,r_1)$ in case $(p_0,\dot{q}_0)\leq_{\mathbb{A}_{\mathrm{Even}(\beta)}\ast\dot{\mathbb{S}}^\pi_\beta} (p_1,\dot{q}_1)$, $\dom(r_1)$ $\subseteq\dom(r_0)$ and for each $\beta\in\dom(r_1)$,
$\sigma^\alpha_\beta(p_0,q_0)\Vdash_{\mathbb{A}_{\mathrm{Even}(\beta)}\ast \dot{\mathbb{S}}^\pi_\beta}\text{$``\dot{r}_0(\beta)\leq \dot{r}_1(\beta)$''}. $
\end{defi}
The proof of the next result is essentially the same as \cite[Lemma 3.13]{FriHon} or \cite[Lemma 3.8]{GolPov}.
\begin{prop}\label{ProjectionRandTruncations}
Let $\alpha\in\mathcal{A}\cap (\beta_0,\lambda^+)$. Then there is a projection between $\mathbb{R}$ and $\mathrm{RO}^+(\mathbb{R}\upharpoonright\alpha)$.
\end{prop}

\begin{prop}\label{TruncationsAndAronszajnTree}
Let $\dot{T}$ be a $\mathbb{R}$-name for a $\lambda$-Aronszajn tree. There is $\beta^*\in\mathcal{A}\cap (\beta_0,\lambda^+)$, such that $V^{\mathbb{R}\upharpoonright\beta^*}\models \text{$``T$ is a $\lambda$-Aronszajn tree''}$ 
\end{prop}
\begin{proof}
Let $\dot{T}$ be a $\mathbb{R}$-name for a $\lambda$-Aronszajn tree $T$. Without loss of generality, $\one_\mathbb{R}\Vdash_\mathbb{R} \dot{T}\subseteq\lambda$. 
Let $\{{A}_\alpha\}_{\alpha<\lambda}$ be a family of maximal antichains deciding $\text{$``\alpha\in \dot{T}$''}$. Set 
${A}^*:=\bigcup_{\alpha\in\lambda} A_\alpha$ and observe that $|A^*|\leq \lambda$. In particular, there is some $\beta^*\in\mathcal{A}\cap (\beta_0,\lambda^+)$ be such that  $\dom(p)\subseteq \kappa\times\beta^*$, for any condition $(p,\dot{q},r)\in A^*$. Clearly $\{A_\alpha\}_{\alpha<\lambda}$ is a family of maximal antichains in $\mathbb{R}\upharpoonright\beta^*$ deciding the same properties, hence $V^{\mathbb{R}\upharpoonright\beta^*}\models \text{$``T$ is $\lambda$-Aronszajn''}$.
\end{proof}
Let $\pi^*:\beta^*\rightarrow \lambda$ be a bijection extending $\pi$. We use $\pi^*$ to define an $\in$-isomorphism between $V^{\mathbb{A}_{\beta^*}}$ and $V^{\mathbb{A}_\lambda}$.
\footnote{This choice will guarantee that our future construction coheres with the previous one.} 
Again, $\mathfrak{U}_\lambda^{\pi^*}:=\pi^*{(\dot{\mathfrak{U}}_{\beta^*})}_{\pi^*[G\upharpoonright\beta^*]}$ is a $\lhd$-increasing sequence of measures 
which (pointwise) extends the sequence $\mathfrak{U}^\pi_{\lambda}$. Similarly, define $\mathfrak{B}^{\pi^*}_\lambda:=\pi^*(\mathfrak{B}_{\beta^*})_{\pi^*[G\upharpoonright\beta^*]}$. 
Let $\mathbb{S}^{\pi^*}_\lambda:=\mathbb{S}_{(\kappa,\mu, \mathfrak{U}^{\pi^*}_\lambda, \mathfrak{B}^{\pi^*}_\lambda)}$. For the ease of notation, let  $H^*$ be the $\mathbb{A}_{\mathrm{Even}(\lambda)}$-generic filter generated by $\pi^*[G\upharpoonright\beta^*]$. 

\begin{prop}\label{R*RtruncatedAreIsomorphic} $ $
\begin{enumerate}
\item There is an isomorphism $\varphi:\mathbb{A}_{\beta^*}\ast \dot{\mathbb{S}}_{\beta^*}\rightarrow\mathbb{A}_{\lambda}\ast \dot{\mathbb{S}}^{\pi^*}_{\lambda}$.
\item For each $\beta\in\mathcal{B}$ the function $\varrho^{\lambda}_{\beta}=\sigma^{\beta^*}_{\beta}\circ \varphi^{-1}$ establishes a projection between $\mathbb{A}_{\lambda}\ast \dot{\mathbb{S}}^{\pi^*}_{\lambda}$ and $\mathrm{RO}^+(\mathbb{A}_{\even(\beta)}\ast \dot{\mathbb{S}}^{\pi}_\beta)$.
\end{enumerate}
\end{prop}
\begin{proof}
For (1), observe that the subposet of $\mathbb{A}_{\beta^*}\ast \dot{\mathbb{S}}_{\beta^*}$ formed by conditions of the form $(p,(\check{g},\dot{H}))$, is dense. Analogously, for $\mathbb{A}_{\lambda}\ast \dot{\mathbb{S}}^{\pi^*}_{\lambda}$. It is routine to check that $(p, (\check{g},\dot{H}))\mapsto (\pi^*(p), (\check{g}, \pi^*(\dot{H})))$ defines an isomorphism between these two dense subposets. Observe that now (2) is immediate as $\sigma^{\beta^*}_\beta$ is a projection.
\end{proof}

\begin{defi}\label{DefinitionR*}
A condition in $\mathbb{R}^*$ is a triple $(p,\dot{q},r)$ for which all the following hold:
\begin{enumerate}
\item $(p,q)\in \mathbb{A}_\lambda \ast \dot{\mathbb{S}}^{\pi^*}_\lambda$;
\item $r$ is a partial function with $\dom(r)\in[\mathcal{B}]^{<\delta}$;
\item  For every $\beta\in\dom(r)$, $r(\beta)$ is a $\mathbb{A}_{\mathrm{Even}(\beta)}\ast \dot{\mathbb{S}}^\pi_\beta$-name such that 
$$\one_{\mathbb{A}_{\mathrm{Even}(\beta)}\ast \dot{\mathbb{S}}^\pi_\beta}\Vdash_{\mathbb{A}_{\mathrm{Even}(\beta)}\ast \dot{\mathbb{S}}^\pi_\beta}\text{$``\dot{r}(\beta)\in\dot{\Add}(\delta,1)$''}.$$
\end{enumerate}
For conditions $(p_0,\dot{q}_0,r_0),(p_1,\dot{q}_1,r_1)$  in $\mathbb{R}^*$ we will write $(p_0,\dot{q}_0,r_0)\leq (p_1,\dot{q}_1,r_1)$ in case $(p_0,\dot{q}_0)\leq_{\mathbb{A}_{\mathrm{Even}(\beta)}\ast \dot{\mathbb{S}}^\pi_\beta}(p_1,\dot{q}_1)$, $\dom(r_1)$ $\subseteq\dom(r_0)$ and for each $\beta\in\dom(r_1)$,
$\varrho^\lambda_\beta(p_0,q_0)\Vdash_{\mathbb{A}_{\mathrm{Even}(\beta)}\ast \dot{\mathbb{S}}^\pi_\beta}\dot{r}_0(\beta)\leq \dot{r}_1(\beta). $
\end{defi}
\begin{prop}\label{R*RtruncatedAreIsomorphic2}
$\mathbb{R}^*$ and $\mathbb{R}\upharpoonright\beta^*$ are isomorphic. In particular, $\mathbb{R}^*$ forces that $\dot{T}$ is a $\lambda$-Aronszajn tree.
\end{prop}
\begin{proof}
It is not hard to check that $(p,\dot{q},r)\mapsto (\varphi(p,\dot{q}),r)$ defines an isomorphism between both forcings.
\end{proof}

Given a weakly compact cardinal $\theta$ the weakly compact filter on $\theta$, $\mathcal{F}_\theta$, is the filter defined by all subsets $X\subseteq\theta$ such that \textit{$\theta\setminus X$ is not $\Pi^1_1$-indescribable in $\theta$}. The filter $\mathcal{F}_\theta$ is proper and normal (see \cite[Proposition 6.11]{Kan}), hence it extends $\mathrm{Club}(\theta)$, and thus concentrates on the set of Mahlo cardinals below $\theta$. 
\begin{lemma}\label{DefofB*}
There is $\mathcal{B}^*\in(\mathcal{F}_\lambda)^V$,  $\mathcal{B}^*\subseteq \mathcal{B}$, with $\delta<\min\mathcal{B}^*$ such that for every $\alpha\in\mathcal{B}^*$, the sequences  $\langle (\dot{U}^{\pi^*}_{\xi})_{H^*}\cap V[H \upharpoonright\alpha]\mid \xi<\mu\rangle$, $\langle (\dot{B}^{\pi^*}_{\xi})_{H^*}\cap V[H^*\upharpoonright\alpha]\mid \xi<\mu\rangle$ are suitable to define Sinapova forcing $V[H^*\upharpoonright\alpha]$.
\end{lemma}
\begin{proof}
The construction of $\mathcal{B}^*$ is the same as for $\mathcal{B}$ but starting from $\mathcal{B}$ instead of $\lambda$. By construction, $\mathcal{B}^*$ is an unbounded set closed by increasing sequences of length $\geq\delta$, hence $\mathcal{B}^*\in (\mathcal{F}_\lambda)^V$.
\end{proof}
\begin{notation}
\rm{
For each $\alpha\in\mathcal{B}^*$, let $\mathfrak{U}^{\pi^*}_\alpha$ and $\mathfrak{B}^{\pi^*}_\alpha$ denote the sequences witnessing Lemma \ref{DefofB*} and set
$\mathbb{S}^{\pi^*}_\alpha:=\mathbb{S}_{(\kappa,\mu, \mathfrak{U}^{\pi^*}_\alpha, \mathfrak{B}^{\pi^*}_\alpha)}$. 
}
\end{notation}

\begin{lemma}\label{ProjectionsPistaronEven}
Let $\hat{\mathcal{B}^*}=\mathcal{B}^*\cup \{\lambda\}$ and $\alpha<\gamma\in\hat{\mathcal{B}}^*$. There are projections
\begin{enumerate}
\item $\varrho^\gamma_\alpha:\mathbb{A}_\gamma\ast \dot{\mathbb{S}}^{\pi^*}_\gamma\rightarrow\mathrm{RO}^+(\mathbb{A}_{\even(\alpha)}\ast\dot{\mathbb{S}}^{\pi}_\alpha)$,
\item $\hat{\varrho}^\gamma_\alpha:\mathrm{RO}^+(\mathbb{A}_\gamma\ast \dot{\mathbb{S}}^{\pi^*}_\gamma)\rightarrow\mathrm{RO}^+(\mathbb{A}_{\even(\alpha)}\ast\dot{\mathbb{S}}^{\pi}_\alpha)$.
\end{enumerate}
Moreover, for each $\alpha<\gamma\in\mathcal{B}^*$, $\varrho^\gamma_\alpha=\sigma^\gamma_\alpha$.
\end{lemma}
\begin{proof}
The construction of $\varrho^\gamma_\alpha$ and $\hat{\varrho}^\gamma_\alpha$ is analogous to Lemma \ref{ProjectionsCohenPart}, again using a suitable version of Proposition~\ref{PropOnProjectionOfGenerics}.  A proof for the moreover part can be found in \cite[Lemma 3.18]{FriHon}.
\end{proof}
The moreover clause of the previous lemma is crucial since it guarantees that there are no disagreements between the projections 
defining $\mathbb{R}^*$ and the projections 
intended to define its truncations.
\begin{defi}[Truncations of $\mathbb{R}^*$]\label{TruncationsR*}
Let $\gamma\in\mathcal{B}^*$. A condition in $\mathbb{R}^*\upharpoonright\gamma$ is a triple $(p,\dot{q},r)$ for which all the following hold:
\begin{enumerate}
\item $(p,q)\in \mathbb{A}_\gamma \ast \dot{\mathbb{S}}^{\pi^*}_\gamma$, 
\item $r$ is a partial function with $\dom(r)\in[\mathcal{B}^*\cap \gamma]^{<\delta}$;
\item For every $\alpha\in\dom(r)$, $r(\alpha)$ is a $\mathbb{A}_{\mathrm{Even}(\alpha)}\ast \dot{\mathbb{S}}^\pi_\alpha$-name such that
$$\one_{\mathbb{A}_{\mathrm{Even}(\alpha)}\ast \dot{\mathbb{S}}^\pi_\alpha}\Vdash_{\mathbb{A}_{\mathrm{Even}(\alpha)}\ast \dot{\mathbb{S}}^\pi_\alpha}\text{$``\dot{r}(\alpha)\in\Add(\delta,1)$''}.$$
\end{enumerate}
For conditions $(p_0,\dot{q}_0,r_0), (p_1,\dot{q}_1,r_1)$ in $\mathbb{R}^*\upharpoonright\gamma$ we will write $(p_0,\dot{q}_0,r_0)\leq (p_1,\dot{q}_1,r_1)$ in case $(p_0,\dot{q}_0)\leq (p_1,\dot{q}_1)$, $\dom(r_1)$ $\subseteq\dom(r_0)$ and for each $\alpha\in\dom(r_1)$,
$\varrho^\gamma_\alpha(p_0,q_0)\Vdash_{\mathbb{A}_{\mathrm{Even}(\alpha)}\ast \dot{\mathbb{S}}^\pi_\alpha}\dot{r}_0(\alpha)\leq \dot{r}_1(\alpha). $
\end{defi}
The proof of the next result is analogous to Proposition \ref{ProjectionRandTruncations}.
\begin{prop}
For each $\gamma\in\mathcal{B}^*$, there is a projection between $\mathbb{R}^*$ and $\mathrm{RO}^+(\mathbb{R}^*\upharpoonright\gamma)$. In particular, $\mathbb{R}^*$ is isomorphic to the iteration $\mathbb{R}^*\upharpoonright\gamma\ast (\mathbb{R}^*/\mathbb{R}^*\upharpoonright\gamma)$.
\end{prop}
\begin{lemma}\label{AronszajnWeakComp}
Assume there is a $\lambda$-Aronszajn tree $T$ in $V^{\mathbb{R}^*}$. Then there is $\gamma\in\mathcal{B}^*$ such that $T\cap \gamma$ is a $\gamma$-Aronszajn tree in $V^{\mathbb{R}^*\upharpoonright\gamma}$.
\end{lemma}
\begin{proof}
Let $\dot{T}$ be a $\mathbb{R}^*$-name such that $\one_{\mathbb{R}^*}\Vdash_{\mathbb{R}^*}\text{``$\dot{T}$ is a $\lambda$-Aronszajn tree''}$. Without loss of generality $\dot{T}$ is a $\mathbb{R}^*$-name for a subset of $\lambda$. It is not hard to check that this is equivalent to a $\Pi^1_1$ sentence $\Phi$ in the language $\mathcal{L}=\{\in,\mathbb{R}^*,\dot{T},\lambda\}$. 
Since $\lambda$ is weakly compact, hence $\Pi^1_1$-indescribable, there is a set $X\in\mathcal{F}_\lambda$ such that for each $\gamma\in X$, $\langle V_\gamma,\in, \mathbb{R}^*\cap V_\gamma, \dot{T}\cap \gamma, \gamma\rangle\models \Phi$. By Lemma \ref{DefofB*} and the former discussion we can assume that all these $\gamma$ are Mahlo and that $\gamma\in\mathcal{B}^*$. In particular, $\mathbb{R}^*\cap V_\gamma=\mathbb{R}^*\upharpoonright \gamma$, and thus $\langle V_\gamma,\in, \mathbb{R}^*\upharpoonright\gamma, \dot{T}\cap \gamma, \gamma\rangle\models \Phi$. Notice that $\Phi$ is absolute between the universe of sets and this structure, hence $\one_{\mathbb{R}^*\upharpoonright\gamma}\Vdash_{\mathbb{R}^*\upharpoonright\gamma}\text{``$\dot{T}\cap \gamma$ is a $\gamma$-Aronszajn tree''}$.
\end{proof}

\begin{lemma}
Assume that there is a $\lambda$-Aronszajn tree $T\s \lambda$ in $V^{\mathbb{R}^*}$. Let $\gamma\in\mathcal{B}^*$ be as in the previous lemma. Then $\mathbb{R}^*/(\mathbb{R}^*\upharpoonright\gamma)$ adds $b_\gamma$, a  cofinal branch throughout $T\cap \gamma$.
\end{lemma}
\begin{proof}
Observe that in $V^{\mathbb{R}^*}$ there is a cofinal branch $b_\gamma$ for $T\cap\gamma$, 
 as $T$ is a $\lambda$-tree. Nonetheless, $T\cap \gamma$ is $\gamma$-Aronszajn in $V^{\mathbb{R}^*\upharpoonright\gamma}$ so that this branch must be added by the quotient   $\mathbb{R}^*/(\mathbb{R}^*\upharpoonright\gamma)$.
\end{proof}
By combining Proposition~\ref{TruncationsAndAronszajnTree} and \ref{R*RtruncatedAreIsomorphic2} with the above lemma  it follows that if the quotients  $\mathbb{R}^*/(\mathbb{R}^*\upharpoonright\gamma)$ do not add $\gamma$-branches then $\TP(\lambda)$ holds in $V[R]$. 
 
In the next series of lemmas we will prove that for each $\gamma\in\mathcal{B}^*$ there are forcings $\mathbb{P}_\gamma$ and $\mathbb{Q}_\gamma$ fulfilling the following properties:
\begin{itemize}
\item[($\alpha_\gamma$)] $\mathbb{P}_\gamma\times\mathbb{Q}_\gamma$ projects onto $\mathbb{R}^*/(\mathbb{R}^*\upharpoonright\gamma)$ in $V^{\mathbb{R}^*\upharpoonright\gamma}$.
\item[($\beta_\gamma$)] $\mathbb{P}_\gamma\times\mathbb{Q}_\gamma$ does not add new branches to $T\cap \gamma$ over $V^{\mathbb{R}^*\upharpoonright\gamma}$. 
\end{itemize}
Combining ($\alpha_\gamma$) and ($\beta_\gamma$) we would conclude that $\mathbb{R}^*/(\mathbb{R}^*\upharpoonright\gamma)$ does not add $\gamma$-branches to $T\cap \gamma$. In particular, if this is  true for each $\gamma\in\mathcal{B}^*$ then $V[R]\models \TP(\lambda)$.
\begin{defi}\label{DefiAfterPropertiesAlphaBeta}
For each $\gamma\in\mathcal{B}^*\cup\{\lambda\}$, define $\mathbb{C}_\gamma:=\mathbb{A}_\gamma\ast\dot{\mathbb{S}}^{\pi^*}_\gamma$ $\mathbb{P}_\gamma:=\mathbb{C}_\lambda/\mathbb{C}_\gamma$ and $\mathbb{U}_\gamma:=\{(\one,\dot{\one},r)\mid (\one,\dot{\one}, r)\in \mathbb{R}^*\upharpoonright\gamma\}$. Also, over $V^{\mathbb{R}^*\upharpoonright\gamma}$, define $\mathbb{Q}_\gamma:=\{(\one,\dot{\one},r)\mid (\one,\dot{\one},r)\in\mathbb{R}^*/\mathbb{R}^*\upharpoonright\gamma\}$.
\end{defi}
Standard arguments  shows that $\mathbb{Q}_\gamma$ is $\delta$-directed closed over $V^{\mathbb{R}^*\upharpoonright\gamma}$. Moreover, arguing as in Proposition~\ref{projectionU}  and Proposition \ref{PropertiesOfV[R]} one obtains the following:
\begin{prop}\label{projectionUgamma}
For each $\gamma\in\mathcal{B}^*\cup\{\lambda\}$, the following hold:
\begin{enumerate}
\item $\mathbb{U}_\gamma$ is $\delta$-directed closed.
\item $\mathbb{C}_\gamma\times\mathbb{U}_\gamma$ projects onto $\mathbb{R}^*\upharpoonright\gamma$ via the map $\langle (p,\dot{q}),(\one,\dot{\one},r)\rangle\mapsto (p,\dot{q},r)$.
\item $V^{\mathbb{C}_\gamma}$ and $V^{\mathbb{R}^*\upharpoonright\gamma}$ have the same ${<}\delta$-sequences
\end{enumerate}
\end{prop}
\begin{prop}
$ $
\begin{enumerate}
\item $\mathbb{R}^*\upharpoonright\gamma$ is $\gamma$-Knaster. In particular, all $V$-cardinals ${\geq}\gamma$ are preserved.
\item $V^{\mathbb{R}^*\upharpoonright\gamma}\models \text{$``\kappa$ is  strong limit  with $\cof(\kappa)=\mu$''}$.
\item $\mathbb{R}^*\upharpoonright\gamma$ collapses all the cardinals in the interval $(\kappa,\delta)\cup (\delta^+,\gamma)$. 
In particular, $V^{\mathbb{R}^*\upharpoonright\gamma}\models\text{ $``\delta=\kappa^+\,\wedge\,\gamma=\kappa^{++}$''}$.
\item $V^{\mathbb{R}^*\upharpoonright\gamma}\models\text{$``2^\kappa\geq \gamma$''}$.
\end{enumerate}
\end{prop}
\begin{prop}
For each $\gamma\in\mathcal{B}^*$, $\mathbb{P}_\gamma\times \mathbb{Q}_\gamma$ satisfies $(\alpha_\gamma)$.
\end{prop}
\begin{proof}
By definition, a condition in $\mathbb{R}^*/\mathbb{R}^*\upharpoonright\gamma$ is a triple $(p,\dot{q},r)$ such that $(\pi^\lambda_\gamma(p,\dot{q}),r\upharpoonright\gamma)\in \mathbb{R}^*\upharpoonright\gamma$, where $\pi^\lambda_\gamma$ is the composition of $\varrho^\lambda_\gamma$ with the standard isomorphism between $\mathbb{C}_\gamma$ and $\mathrm{RO}^+(\mathbb{C}_\gamma)$. In particular, $(p,\dot{q})\in \mathbb{P}_\gamma$. Now, it is immediate to check that $\tau: \mathbb{P}_\gamma\times\mathbb{Q}_\gamma\rightarrow\mathbb{R}^*/\mathbb{R}^*\upharpoonright\gamma$ given by $\langle(p,\dot{q}),(\one,\one,r)\rangle\mapsto (p,\dot{q},r)$ defines a projection. 
\end{proof}
It thus remains to prove that $\mathbb{P}_\gamma\times\mathbb{Q}_\gamma$ satisfies $(\beta_\gamma)$.
\begin{prop}
Let $\gamma\in\mathcal{B}^*$. If $\mathbb{P}_\gamma\times\mathbb{P}_\gamma$ is $\delta$-cc over $V^{\mathbb{C}_\gamma}$ then $\mathbb{P}_\gamma\times\mathbb{Q}_\gamma$ witnesses $(\beta_\gamma)$.
\end{prop}
\begin{proof}
Let us first prove that if $\mathbb{P}_\gamma\times\mathbb{P}_\gamma$ is $\delta$-cc over $V^{\mathbb{R}^*\upharpoonright\gamma}$ then $\mathbb{P}_\gamma\times\mathbb{Q}_\gamma$ witnesses $(\beta_\gamma)$. Notice that $\one_{\mathbb{Q}_\gamma}\Vdash^{V^{\mathbb{R}^*\upharpoonright\gamma}}_{\mathbb{Q}_\gamma}\text{$``|\gamma|=\delta$''}$. Since $\mathbb{Q}_\gamma$ is $\delta$-directed closed, Easton's Lemma (see e.g. \cite[Lemma 4.4.]{GolMoh}) yields $\one_{\mathbb{Q}_\gamma}\Vdash^{V^{\mathbb{R}^*\upharpoonright\gamma}}_{\mathbb{Q}_\gamma}\text{$``\mathbb{P}_\gamma\times\mathbb{P}_\gamma$ is $\delta$-cc''}$. Now by appealing to \cite[Lemma 2.2]{Ung} it follows that $\mathbb{Q}_\gamma$ forces, over $V^{\mathbb{R}^*\upharpoonright\gamma}$  that $\mathbb{P}_\gamma$ does not add a cofinal branch to $T\cap\gamma.$ On the other hand, $\one_{\mathbb{R}^*\upharpoonright\gamma}\Vdash_{\mathbb{R}^*\upharpoonright\gamma}\text{$``2^\kappa\geq\gamma$''}$ and  $\one_{\mathbb{R}^*\upharpoonright\gamma}\Vdash_{\mathbb{R}^*\upharpoonright\gamma}\text{``$\mathbb{Q}_\gamma$ is $\kappa^+$-closed''}$, so by Silver's theorem \cite[Lemma V.2.26]{Kun}, $\mathbb{Q}_\gamma$ does not add  cofinal branches to $T\cap \gamma$. Finally 
we use Proposition~\ref{projectionUgamma} (3) to 
infer that if $\mathbb{P}_\gamma\times\mathbb{P}_\gamma$ is $\delta$-cc over $V^{\mathbb{C}_\gamma}$ then it is also $\delta$-cc over $V^{\mathbb{R}^*\upharpoonright\gamma}$. 
\end{proof}

\begin{lemma}\label{PreviousNotInTheQuotient}
Let $\mathbb{P}$ and $\mathbb{Q}$ be  two forcing notions and $\pi:\mathbb{P}\rightarrow\mathbb{Q}$ be a projection. For every $p\in\mathbb{P}$ and $q\in\mathbb{Q}$, $q\Vdash_\mathbb{Q} p\notin \dot{(\mathbb{P}/\mathbb{Q})}$ if and only if for every generic filter $G\subseteq\mathbb{P}$ with $p\in G$, $q$ is not in $H$, the generic filter generated by $\pi[ G]$. In particular, if $\pi(p)\perp q$, $q\Vdash_{\mathbb{Q}} p\notin \dot{(\mathbb{P}/\mathbb{Q})}$.
\end{lemma}
\begin{proof}
The first implication is obvious. Conversely, assume that there is $q'\leq_\mathbb{Q} q$ be such that $q'\Vdash_\mathbb{Q} p\in \dot{(\mathbb{P}/\mathbb{Q})}$. Let $H\s\mathbb{Q}$ be some generic filter over $V$ containing $q$. Hence, $p\in\mathbb{P}/H$. Now let $G\s\mathbb{P}/H$ be some generic filter over $V[H]$ containing $p$. Clearly $\pi[G]=H$ and $q\in H$, which yields the desired contradiction.

\end{proof}

\begin{remark}
\rm{
Let $\gamma\in\mathcal{B}^*\cup\{\lambda\}$. Observe that $\tilde{\mathbb{C}}_\gamma:=\{(p,(\check{g},\dot{H}))\mid p\in\mathbb{A}_\gamma,\, g\in V,\, p\Vdash_{\mathbb{A}_\gamma} (\check{g},\dot{H})\in\dot{\mathbb{S}}^{\pi^*}_\gamma\}$ endowed with the induced order is a dense subposet of $\mathbb{C}_\gamma$. Thus, for our current purposes it is enough to assume that $\mathbb{C}_\gamma=\tilde{\mathbb{C}}_\gamma$.
}
\end{remark}

\begin{notation}
\rm{
For each $\gamma\in\mathcal{B}^*\cup\{\lambda\}$, set $g(\mu):=\varepsilon$ and $\kappa_{g(\mu)}:=\kappa$, for every $g$ which is a stem for some $q\in \dot{\mathbb{S}}^{\pi^*}_\gamma$. Observe that $\mathcal{P}_{\kappa_{g(\mu)}}(\kappa_\eta\cap g(\mu))=\mathcal{P}_\kappa(\kappa_\eta)$, for each $\eta<\mu$.




}
\end{notation}

\begin{convention}
\rm{For the ease of notation --and provided no confusion arise-- we shall tend to omit the mention to the particular family of measures that we are working with. For instance, instead of writting $(\mathfrak{U}^{\pi^*}_\gamma)^{\xi}_{\eta, x}$ we shall simply write $U^\xi_{\eta,x}$.
} 
\end{convention}

\begin{lemma}\label{NotNotInTheQuotient}
Let $\gamma\in\mathcal{B}^*$, $r=(p,(\check{h}, \dot{H}))\in\mathbb{C}_\lambda$ and $r'=(q, (\check{f},\dot{F}))\in\mathbb{C}_\gamma$. Then, $r'\Vdash_{\mathbb{C}_\gamma} \text{$``r\notin\mathbb{P}_\gamma$''}$ if and only if one of the following hold:
\begin{enumerate}
\item $p\upharpoonright\gamma\perp_{\mathbb{A}_\gamma} q$;
\item $p\upharpoonright\gamma\parallel_{\mathbb{A}_\gamma} q$ and $h\cup f $ is not a $\prec$-increasing function;
\item $p\upharpoonright\gamma\parallel_{\mathbb{A}_\gamma} q$, $h\cup f $ is a $\prec$-increasing function and 
$$p\cup q\Vdash_{\mathbb{A}_\lambda} (\check{f},\dot{F}){}^\curvearrowright (\check{h}\setminus \check{f})\notin\dot{\mathbb{S}}^{\pi^*}_\gamma\,\vee\, (\check{h},\dot{H}){}^\curvearrowright (\check{f}\setminus \check{h})\notin\dot{\mathbb{S}}^{\pi^*}_\lambda.\footnote{Here we are identifying the $\mathbb{A}_\gamma$-name $\mathbb{S}^{\pi^*}_\gamma$ with its standard extension to a $\mathbb{A}_\lambda$-name.}$$
\end{enumerate}
\end{lemma}
\begin{proof}
First, observe that two conditions $(h,H), (f,F)\in \mathbb{S}^{\pi^*}_\lambda$ are compatible if and only if $h\cup f$ is a $\prec$-increasing function and $(h,H){}^\curvearrowright (f\setminus h), (f,F){}^\curvearrowright (h\setminus f)\in\mathbb{S}^{\pi^*}_\lambda$. Thereby, if some of the above conditions is true, $\varrho^\lambda_\gamma(r)\perp_{\mathbb{C}_\gamma} r'$. Thus, Lemma \ref{PreviousNotInTheQuotient} yields $r'\Vdash_{\mathbb{C}_\gamma} \text{$``r\notin\mathbb{P}_\gamma$''}$. Conversely, assume that (1)-(3) are false. Since (1) and (2) are false,  $p\cup q\in \mathbb{A}_\lambda$ and $i:=f\cup h$ is $\prec$-increasing. Also, since (3) is false, we may let  a condition $a\leq_{\mathbb{A}_\lambda} p\cup q$ forcing the opposite. Let $A\s\mathbb{A}_\lambda$ generic (over $V$) containing $a$. By the above, in $V[A]$, $(f,F){}^\curvearrowright(h\setminus f)\in \mathbb{S}^{\pi^*}_\gamma$ and  $(h,H){}^\curvearrowright(f\setminus h)\in \mathbb{S}^{\pi^*}_\lambda$, hence both Sinapova conditions are compatible. Let $(i,I)\in\mathbb{S}^{\pi^*}_\lambda$  be a condition witnessing this compatibility and $S\s\mathbb{S}^{\pi^*}_\lambda$ generic (over $V[A]$) containing $(i,I)$. Set $r^*:=(a,(\check{i},\dot{I}))$. Clearly, $r^*\in A\ast\dot{S}$ and $r^*\leq_{\mathbb{C}_\lambda} r$, so $r\in A\ast\dot{S}$. On the other hand, $\varrho^\lambda_\gamma[A\ast\dot{S}]$ generates a $\mathbb{C}_\gamma$-generic filter containing $r'$, hence Lemma \ref{PreviousNotInTheQuotient} yields  $r'\nVdash_{\mathbb{C}_\gamma} \text{$``r\notin\mathbb{P}_\gamma$''}$, as wanted.


\end{proof}
For each $\gamma\in\mathcal{B}^*\cup\{\lambda\}$, and unless otherwise stated, we will assume that for each $r=(q, (\check{f},\dot{F}))\in \mathbb{C}_\gamma$, $q\Vdash_{\mathbb{A}_\gamma}\text{$``(\check{f},\dot{F})$ is pruned''}$. This is of course feasible by virtue of  Proposition \ref{ExistsPrunedExtension}.
\begin{lemma}\label{RefinningWithRowbottom}
Let $\gamma\in\mathcal{B}^*$, $r=(p,(\check{h}, \dot{H}))\in\mathbb{C}_\lambda$ and $r'=(q, (\check{f},\dot{F}))\in\mathbb{C}_\gamma$. Assume that $q\leq_{\mathbb{A}_\gamma} p\upharpoonright\gamma$, $h\s f$ and $$(\Upsilon)\;\; p\cup q\Vdash_{\mathbb{A}_\lambda}\text{$`` \forall\theta\in\dom(\dot{H})\, \left(\dot{H}(\theta)\cap \dot{\mathcal{P}}_{\kappa_{\check{f}(\tau_\theta)}}(\kappa_\theta\cap \check{f}(\tau_\theta))\in \dot{U}^\theta_{\tau_\theta, \check{i}(\tau_\theta)}\right)$''},$$
where $q\cup p\Vdash_{\mathbb{A}_\lambda} \text{$``\tau_\theta=r_{\dom(\check{f})}(\check{\theta})$''}$.
Then there is a $\mathbb{A}_\gamma$-name $\dot{I}$  for which all the following hold:
\begin{enumerate}
\item[(I)] $q\Vdash_{\mathbb{A}_\gamma}\text{$``(\check{f},\dot{I})\leq_{\mathbb{S}^{\pi^*}_\gamma} (\check{f},\dot{F})\, \wedge\, (\check{f},\dot{I})$ is pruned''}$.
\item[(II)] $q\Vdash_{\mathbb{A}_\gamma}\text{$``\forall \tau\in [\prod_\xi \dot{I}(\xi)]^{<\omega}\, \left( p\nVdash_{\mathbb{A}_\lambda/\mathbb{A}_\gamma} (\check{h},\dot{H}){}^\curvearrowright \tau\notin\dot{\mathbb{S}}^{\pi^*}_\lambda \right)$''}$.
\end{enumerate}
\end{lemma}
\begin{proof}
Let us work over $V^{\mathbb{A}_\gamma\downarrow q}$.  Let $c:[\prod_{\xi}F(\xi)]\rightarrow 2$ be defined as
$$
c(i):=\begin{cases}
0, & \text{if $p\Vdash_{\mathbb{A}_\lambda/\mathbb{A}_\gamma} (\check{h},\dot{H}){}^\curvearrowright i\notin \dot{\mathbb{S}}^{\pi^*}_\lambda$};\\
1, & \text{if $p\nVdash_{\mathbb{A}_\lambda/\mathbb{A}_\gamma} (\check{h},\dot{H}){}^\curvearrowright i\notin \dot{\mathbb{S}}^{\pi^*}_\lambda$}.
\end{cases}
$$
By Lemma \ref{DiagonalSupercompactRowottom} there is  $I\s F$ a  suitable function for $\langle f\rangle$ and homogeneous for $c$.  In particular, $(f,I)\leq_{\mathbb{S}^{\pi^*}_\gamma} (f,F)$ and $(f,I)$ is pruned, as $(f,F)$ was. Thus, (I)  holds. Towards a contradiction, assume that (II) is false. Let $r\leq_{\mathbb{A}_\gamma} q$ be such that $r$ forces the negation of the above formula. By shrinking $r$ we may assume that there is a $\prec$-increasing function $i$ such that $r\Vdash_{\mathbb{A}_\gamma} \check{i}\in [\prod_\xi\dot{I}(\xi)]^{<\omega}$ and $r\Vdash_{\mathbb{A}_\gamma}\text{$``\left( p\Vdash_{\mathbb{A}_\lambda/\mathbb{A}_\gamma} (\check{h},\dot{H}){}^\curvearrowright \check{i}\notin\dot{\mathbb{S}}^{\pi^*}_\lambda \right) $''} $. Since $r\leq_{\mathbb{A}_\gamma} q$, $r\cup p\in\mathbb{A}_\lambda$, hence $r\cup p\Vdash_{\mathbb{A}_\lambda} (\check{h},\dot{H}){}^\curvearrowright \check{i}\notin\dot{\mathbb{S}}^{\pi^*}_\lambda$. Now, since $r$ forces $\dot{I}$ to be homogenous for $\dot{c}$, it follows that for all $j$ with the same domain as $i$, $r\cup p\Vdash_{\mathbb{A}_\lambda} (\check{h},\dot{H}){}^\curvearrowright \check{j}\notin\dot{\mathbb{S}}^{\pi^*}_\lambda$. Since $p$ forces $(\check{h},\dot{H})$ to be pruned  the only chance for this property to hold is that $r\cup p \Vdash_{\mathbb{A}_\lambda} \prod_{\theta\in\dom(i)}\dot{I}(\theta)\cap \prod_{\theta\in\dom(i)}\dot{H}(\theta)=\emptyset$.  Let us show that this is impossible.

Let $\theta\in\dom(i)$. If $\theta>\max(\dom(f))$, $\dot{I}(\theta)$ and $\dot{H}(\theta)$ are names for sets in the measure $U_\theta$, and thus they are not forced to be disjoint. Otherwise, if $\theta<\max(\dom(f))$, since  $r\cup p\leq_{\mathbb{A}_\lambda} q\cup p$ and $(\Upsilon)$ holds, we may find $s\leq_{\mathbb{A}_\lambda} r\cup p$, such that $s\Vdash_{\mathbb{A}_\lambda} \dot{H}(\theta)\cap \dot{\mathcal{P}}_{\kappa_{\check{f}(\tau_\theta)}}(\kappa_\theta\cap \check{f}(\tau_\theta))\in \dot{U}^\theta_{\tau_\theta, \check{f}(\tau_\theta)}$. In particular, $s\Vdash_{\mathbb{A}_\lambda}\dot{I}(\theta)\cap \dot{H}(\theta)\cap \dot{\mathcal{P}}_{\kappa_{\check{f}(\tau_\theta)}}(\kappa_\theta\cap \check{f}(\tau_\theta))\in \dot{U}^\theta_{\tau_\theta, \check{f}(\tau_\theta)}$. Altogether, this produces the desired contradiction.
\end{proof}

\begin{lemma}\label{LyingWithinTheQuotient}
Let $\gamma\in\mathcal{B}^*$, $r=(p,(\check{h}, \dot{H}))\in\mathbb{C}_\lambda$ and $r'=(q, (\check{f},\dot{F}))\in\mathbb{C}_\gamma$. Assume that 
\begin{enumerate}
\item[$(\aleph)$] $q\leq_{\mathbb{A}_\gamma} p\upharpoonright\gamma$;
\item[$(\beth)$] $h\s f$;
\item[$(\gimel)$] $p\cup q\Vdash_{\mathbb{A}_\lambda}\text{$`` (\check{h},\dot{H}){}^\curvearrowright (\check{f}\setminus \check{h})\in\dot{\mathbb{S}}^{\pi^*}_\lambda$''}.$
\end{enumerate}
Let $\dot{I}$ be the function obtained from Lemma \ref{RefinningWithRowbottom} with respect to $r$ and $r'$. Then, $(q,(\check{f},\dot{I}))\Vdash_{\mathbb{C}_\gamma} (p,(\check{h},\dot{H}))\in \mathbb{P}_\gamma$.
\end{lemma}
\begin{proof}
Otherwise, let $r^*:=(r,(\check{j},\dot{J}))\leq_{\mathbb{C}_\gamma}(q,(\check{f},\dot{I}))$ forcing the opposite. By using Lemma \ref{NotNotInTheQuotient} with respect to $r^*$ and $r$ it follows that some of the conditions (1)-(3) must hold. It is not hard to check that $(\aleph)$-$(\gimel)$ implies that (3) holds: particularly, that $r\cup p\Vdash_{\mathbb{A}_\lambda}\text{$`` (\check{h},\dot{H}){}^\curvearrowright (\check{j}\setminus \check{h})\notin\dot{\mathbb{S}}^{\pi^*}_\lambda$''}$ holds. By $(\gimel)$ and since $r\cup p\leq_{\mathbb{A}_\lambda} p\cup q$, $r\cup p\Vdash_{\mathbb{A}_\lambda}\text{$`` (\check{h},\dot{H}){}^\curvearrowright (\check{j}\setminus \check{f})\notin\dot{\mathbb{S}}^{\pi^*}_\lambda$''}.$ Clearly, $r\leq_{\mathbb{A}_\gamma} q$ and $r\Vdash_{\mathbb{A}_\gamma} \check{j}\setminus \check{f}\in[\prod_\xi \dot{I}(\xi)]$. Observe that $(\gimel)$ yields $(\Upsilon)$ of  Lemma \ref{RefinningWithRowbottom}, and this latter implies $r\cup p\nVdash_{\mathbb{A}_\lambda}\text{$`` (\check{h},\dot{H}){}^\curvearrowright (j\setminus f)\notin\dot{\mathbb{S}}^{\pi^*}_\lambda$''}.$  This produces the desired contradiction.
\end{proof}
\begin{remark}
\rm{As the referee has pointed out, there is a somewhat simpler way to prove the above lemma without relying on Lemma \ref{RefinningWithRowbottom}. Let $A$ a $\mathbb{A}_\gamma$-generic with $q\in A$. In $V[A]$ appeal to the Prikry property of $\mathbb{S}^{\pi^*}_\gamma$ and find $(f,I)\leq^* (f,F)$ with $(f,I)\parallel^{V[A]}_{\mathbb{S}^{\pi^*}_\gamma} (p,(\check{h},\dot{H}))\in\mathbb{P}_\gamma$. Now observe that (1)-(3) of Lemma \ref{NotNotInTheQuotient} hold, hence  $(q,(\check{f},\dot{I}))\Vdash_{\mathbb{C}_\gamma}(p,(\check{h},\dot{H}))\in\mathbb{P}_\gamma$. } 
\end{remark}

\begin{lemma}\label{LemmaCCness}
Let $\gamma\in\mathcal{B}^*$, $(q,(\check{f}, \dot{F}))\in\mathbb{C}_\gamma$ and $\dot{r}_0, \dot{r}_1$ be two $\mathbb{C}_\gamma$-names forced by $\one_{\mathbb{C}_\gamma}$ to be  in $\mathbb{P}_\gamma$. Then, there are $(q^*,(\check{f}^*, \dot{F}^*))\in\mathbb{C}_\gamma$,  $(p_0, (\check{h}_0,\dot{H}_0))$, $(p_1, (\check{h}_1,\dot{H}_1))\in \mathbb{P}_\gamma$ and $\bar{p}_0, \bar{p}_1\in \mathbb{A}_\lambda$ be such that the following hold: For  $i\in\{0,1\}$, 
\begin{itemize}
\item[(a)] $(q^*,(\check{f}^*, \dot{F}^*))\leq_{\mathbb{C}_\gamma}(q,(\check{f}, \dot{F}))$,
\item[$(b_i)$] $(q^*,(\check{f}^*, \dot{F}^*))\Vdash_{\mathbb{C}_\gamma}\text{$``\dot{r}_i=(p_i, (\check{h}_i,\dot{H}_i))\in\mathbb{P}_\gamma$''}$,
\item[$(c_i)$] $\bar{p}_i\leq_{\mathbb{A}_\lambda} p_i$ and $(q^*,(\check{f}^*, \dot{F}^*))$ and $(\bar{p}_i, (\check{h}_i,\dot{H}_i))$ satisfy conditions (1)-(3) of Lemma \ref{LyingWithinTheQuotient}.
\end{itemize}
\end{lemma}
\begin{proof}
Let $(q^*,(\check{f}^*, \dot{F}^*))\leq_{\mathbb{C}_\gamma}(q,(\check{f}, \dot{F}))$ and $(p_0, (\check{h}_0,\dot{H}_0))$, $(p_1, (\check{h}_1,\dot{H}_1))\in \mathbb{P}_\gamma$ be such that $(b_0)$ and $(b_1)$ hold. By extending $q^*$ and $f^*$ if necessary, we may further assume that $q^*\leq_{\mathbb{A}_\gamma} p_0\upharpoonright\gamma\cup p_1\upharpoonright\gamma$ and $h_0\cup h_1\s f^*$. For each $i\in\{0,1\}$, combining this with Lemma \ref{NotNotInTheQuotient} it follows that condition (4) must fail. Thus, there is $\bar{p}_i\leq_{\mathbb{A}_\lambda} q^*\cup p_i$ with  $\bar{p}_i\Vdash_{\mathbb{A}_\lambda} (\check{h}_i,\dot{H}_i){}^\curvearrowright (\check{f}^*\setminus\check{h}_i)\in \dot{\mathbb{S}}^{\pi^*}_\lambda$. Again, extend $p^*$ to ensure $q^*\leq_{\mathbb{A}_\gamma} \bar{p}_0,\bar{p}_1$. It should be  clear at this point that, for $i\in\{0,1\}$, $(q^*,(\check{f}^*, \dot{F}^*))$ and $(\bar{q}_i,(\check{h}_i,\dot{H}_i))$ witness $(c_i)$.
\end{proof}
Finally, we are in conditions to prove the $\delta$-ccness of  $\mathbb{P}_\gamma\times\mathbb{P}_\gamma$.
\begin{lemma}
Let $\gamma\in\mathcal{B}^*$. Then, $\one_{\mathbb{C}_\gamma}\Vdash_{\mathbb{C}_\gamma} \text{$``\mathbb{P}_\gamma\times\mathbb{P}_\gamma$ is $\delta$-cc''}$.
\end{lemma}
\begin{proof}
Let $\{(\dot{r}^0_\alpha,\dot{r}^1_\alpha)\}_{\alpha<\delta}$ be a collection of $\mathbb{C}_\gamma$-names that $\one_{\mathbb{C}_\gamma}$ forces to be in a  maximal antichain of $\mathbb{P}_\gamma\times\mathbb{P}_\gamma$. Appealing to Lemma \ref{LemmaCCness}  we find families $\{(q^*_\alpha, (\check{f}^*_\alpha, \dot{F}^*_\alpha))\}_{\alpha<\delta}$, $\{\langle (p^0_\alpha, (\check{h}^0_\alpha, \dot{H}^0_\alpha)),(p^1_\alpha, (\check{h}^1_\alpha, \dot{H}^1_\alpha))\rangle\}_{\alpha<\delta}$ and\linebreak $\{\langle \bar{p}^0_\alpha, \bar{p}^1_\alpha\rangle\}_{\alpha<\delta}$ witnessing it.

It is not hard to check that for each $\varrho\in\mathcal{B}^*\cup\{\lambda\}$, $\mathbb{C}_\varrho$ is $\delta$-Knaster, hence $\mathbb{C}_\gamma\times\mathbb{C}_\lambda^2$ also.  In particular, $\mathbb{C}_\gamma\times\mathbb{C}_\lambda^2$ is $\delta$-cc, and thus we may assume that all the above conditions are compatible. Modulo a further refinement, we may also assume that $f^*_\alpha=f^*$, $h^0_\alpha=h^0$ and $h^1_\alpha=h^1$, for each $\alpha<\delta$. For each $\alpha<\beta<\delta$,  set $r_{\alpha,\beta}:=(q^*_\alpha\cup q^*_{\beta}, (f^*, \dot{F}^*_\alpha\wedge\dot{F}^*_{\beta}))$ and $r'_{i,\alpha,\beta}:=(\bar{p}^i_\alpha\cup \bar{p}^i_{\beta}, (h^i, \dot{H}^i_\alpha\wedge\dot{H}^i_{\beta}))$. It is routine to check that, for each $i\in\{0,1\}$, $r_{\alpha,\beta}$ and $r'_{i,\alpha,\beta}$ witness the hypotheses of Lemma \ref{LyingWithinTheQuotient}, hence there is  $r^*_{\alpha,\beta}\leq_{\mathbb{C}_\gamma}r_{\alpha,\beta}$ forcing that both $r'_{0,\alpha,\beta}$ and $r'_{1,\alpha,\beta}$ are in $\mathbb{P}_\gamma$.  In particular, $r_{\alpha,\beta}^*\Vdash_{\mathbb{C}_\gamma}(\dot{r}^0_\alpha, \dot{r}^1_\alpha)\parallel_{\mathbb{P}_\gamma\times\mathbb{P}_\gamma}(\dot{r}^0_{\beta}, \dot{r}^1_{\beta})$, which entails the desired contradiction.

\end{proof}

\section{$\TP(\kappa^+)$ holds}\label{TPkappa+Section}
In this section we conclude the proof of Theorem \ref{MainTheorem1} by showing that $\TP(\kappa^+)$ holds in $V[R]$. Once again, we 
only give details when $\Theta=\lambda^+$, as the more general case is completely parallel. In essence the arguments exposed here are due to Sinapova \cite{SinTree} and Neeman \cite{Nee}.
The only reason in favour of presenting them is to point out some subtle differences between their argument and ours. Also, by showing explicitly the arguments, we hope to convince the skeptic reader
that similar ideas indeed do the job in our context.  To avoid repetitions, we sometimes tend to sketch the main ideas and refer the reader to \cite{SinTree}, \cite{SinUncCof} or \cite{Nee} for more details. The proof of $V[R]\models \TP(\delta)$, at least as conceived in \cite{SinTree},  uses a family of intermediate forcings between $\mathbb{R}$ and $\bar{\mathbb{R}}$ (see Section \ref{SectionTheMainConstruction}). These forcings  $\mathbb{R}_{\dot{p}}$  have the particularity that its generics $R_{\dot{p}}$ resemble $R$. For the record of the section let us recall that $G$, $S$ and $R$ are, respectively, the generic filters for $\mathbb{A}_{\lambda^+}$, $\mathbb{S}_{\lambda^+}$ and $\mathbb{R}$ considered at Section \ref{SectionTheMainConstruction}.
\begin{convention}
\rm{
For each $\mathbb{A}_{\lambda^+}$-name $\dot{q}$ for a condition in  $\mathbb{S}_{\lambda^+}$, we shall denote by $q$ its interpretation by $G$. Also, set $\hat{q}:=\langle (\one, \dot{q}), (\one,\dot{\one},\one)\rangle$ and $q^*:=(\one,\dot{q},\one)$.}
\end{convention}
\begin{defi}
Let $\dot{q}$ be a $\mathbb{A}_{\lambda^+}$-name for a condition in $\mathbb{S}_{\lambda^+}$. Let $\mathbb{R}_{\dot{q}}$  be the set of $(p,\dot{q}',r)\in\mathbb{R}$ endowed with the order
$(p_1,\dot{q}_1,r_1)\leq_{\mathbb{R}_{\dot{p}}}(p_2,\dot{q}_2,r_2)$ if and only if $(p_1,\dot{q}_1)\leq_{\mathbb{A}_{\lambda^+}\ast \dot{\mathbb{S}}_{\lambda^+}}(p_2,\dot{q}_2)$, $\dom(r_2)\subseteq \dom(r_1)$ and for each $\gamma\in\dom(r_2)$, 
$\sigma^{\lambda^+}_\gamma(p_1,\dot{q})\Vdash_{\mathbb{A}_{\even(\gamma)}\ast \dot{\mathbb{S}}^\pi_\gamma}\text{$``\dot{r}_1(\gamma)\leq_{\dot{\Add}(\delta,1)}\dot{r}_2(\gamma)$''}.$
\end{defi}

The next proposition shows that there is a system of projections between the forcings $\bar{\mathbb{R}}$, $\mathbb{R}$ and $\mathbb{R}_{\dot{q}}$ (see \cite[\S2]{SinTree} for details).
\begin{prop}\label{ProjectionsBetweenForcings}
Let $\dot{q}$ be a $\mathbb{A}_{\lambda^+}$-name for a condition in $\mathbb{S}_{\lambda^+}$.
\begin{enumerate}
\item The map $\langle (p,\dot{t}),(\one,\one,r)\rangle\mapsto (p,\dot{t},r)$ defines a projection between $\bar{\mathbb{R}}$ and $\mathbb{R}_{\dot{q}}$ and also between $\bar{\mathbb{R}}\downarrow\hat{q}$ and $\mathbb{R}_{\dot{q}}\downarrow q^*$.
\item The identity entails a projection between  $\mathbb{R}_{\dot{q}}\downarrow q^*$ and $\mathbb{R}\downarrow q^*$. 
\end{enumerate}
Let $q,t$ be conditions in $\mathbb{S}_{\lambda^+}$ such that $t\leq_{\mathbb{S}_{\lambda^+}} q$. Then the identity establishes a projection between $\mathbb{R}_{{q}}$ and $\mathbb{R}_{{t}}$.
\end{prop}

\begin{defi}
 Work in $V[G]$. For each $q\in\mathbb{S}$ define the forcing $\mathbb{U}_{q}$ whose conditions are all $r\in \mathbb{U}$ such that $r_1\leq_{\mathbb{U}_{p}} r_2$ if and only if $\dom(r_2)\subseteq\dom(r_1)$ and there is $p\in G$ such that for each $\gamma\in\dom(r_2)$,
$$\sigma^{\lambda^+}_{\gamma}(p,\dot{q}){\Vdash_{\mathbb{A}_{\even(\gamma)}\ast\dot{\mathbb{S}}^\pi_\gamma}}\text{$``r_1(\gamma)\leq r_2(\gamma)$''}.$$
\end{defi}
The next lemma corresponds with \cite[Lemma 2.7]{SinTree}.
\begin{lemma}
Let $\dot{q}$ be a $\mathbb{A}_{\lambda^+}$-name for a condition in ${\mathbb{S}_{\lambda^+}}$. Then $\mathbb{R}_{\dot{q}}$ and $\mathbb{A}_{\lambda^+}\ast (\dot{\mathbb{S}}_{\lambda^+}\times \dot{\mathbb{U}}_{q})$ are isomorphic. In particular, in $V[G]$, there is a projection between $\mathbb{R}_{{q}}$ and $\mathbb{U}_q$.
\end{lemma} 
\begin{prop}\label{ProjectionbetweenQs}
Work in $V[G]$. For each condition $q\in\mathbb{S}$, the identity yields a projection between $\mathbb{U}$ and $\mathbb{U}_q$. Moreover, for each $t\leq_{\mathbb{S}_{\lambda^+}} q$ the same holds between $\mathbb{U}_q$ and $\mathbb{U}_{t}$.
\end{prop} 
Let $\bar{R}\subseteq \bar{\mathbb{R}}$ a generic filter   whose respective projections onto $\mathbb{R}$, $\mathbb{A}_{\lambda^+}$ and $\mathbb{S}_{\lambda^+}$ induce $R$, $G$ and $S$.\footnote{Recall that these are the generic filters of Section \ref{SectionTheMainConstruction}} 
Let $U\subseteq\mathbb{U}$ be the generic filter induced by $\bar{R}$.  We also need generics for the family  $\langle \mathbb{R}_{p}, \mathbb{U}_p\mid p \in\mathbb{S}\rangle$. For this, we will use the following standard lemma. For a proof see, for instance, \cite[Proposition 4.7]{Ung}.
\begin{lemma}\label{projectionquotient}
Let $\mathbb{P},\mathbb{Q},\mathbb{C}$ be posets and $\pi:\mathbb{P}\rightarrow\mathbb{Q}$ and $\sigma:\mathbb{Q}\rightarrow\mathbb{C}$ be projections. For any generic filter $H\subseteq\mathbb{C}$, the restriction $\pi\upharpoonright\mathbb{P}/H$ is a projection between $\mathbb{P}/H$ and $\mathbb{Q}/H$ in $V[H]$.
\end{lemma} 

For $q\in S$, $q^*\in R$, hence $R\downarrow q^*$ is a generic filter for $\mathbb{R}\downarrow q^*$. Since there are projections $\pi_q$ between $\bar{\mathbb{R}}\downarrow\hat{q}$ and $\mathbb{R}_{\dot{q}}\downarrow q^*$ and $\pi^q$ between $\mathbb{R}_{\dot{q}}\downarrow q^*$ and $\mathbb{R}\downarrow q^*$, the previous lemma ensures that $\pi_q\upharpoonright \bar{\mathbb{R}}/R$ is a projection between $\bar{\mathbb{R}}/R$ and $\mathbb{R}_{\dot{q}}/R$. For each $q\in S$,  let $R_{q}\subseteq \mathbb{R}_{\dot{q}}\downarrow q^*$ be the generic filter over $V[R]$ induced by $\bar{R}$ and $\pi_q$. Analogously, let $U_q \subseteq \mathbb{U}_q$ be the generic filter over $V[G]$ induced by $R_q$ and the corresponding projection. 
\begin{remark}\label{RemarkInclusionGenerics}
\rm{$ $
\begin{enumerate} 
\item By Proposition  \ref{ProjectionsBetweenForcings}, $R_{q}\subseteq R_{q'}\subseteq R$, for each $q'\leq_{\mathbb{S}_{\lambda^+}} q$ in $S$. 
Moreover, for each $s\in \bar{\mathbb{R}}/R$, there is $p\in S$ such that $s\in R_p$ (see \cite[Lemma 3.8]{Sin}).
\item By Proposition \ref{ProjectionbetweenQs}, $U\subseteq U_q\subseteq {U}_{q'}$, each $q'\leq_{\mathbb{S}_{\lambda^+}}  q$ in $S$.
\end{enumerate}
}
\end{remark}
Aiming for a contradiction, assume that $V[R]\models \neg \TP(\delta)$ and let  a $\delta$-tree $(T, <_T)\in V[R]$ be  witnessing this. For each $\alpha<\delta$, set $T_\alpha:=\{u\in T\mid \mathrm{level}(u)=\alpha\}$. Modulo isomorphism, we may assume $T_\alpha=\{\alpha\}\times\kappa$, for each $\alpha<\delta$.  Let $\tau\in V[G]$ be a ${\mathbb{R}}/G$-name for $T$ and assume that $\one_{\mathbb{R}/G}\Vdash_{\mathbb{R}/G}\text{$``\tau$ is a $\delta$-tree''}$. Analogously, let $\dot{T}\in V[G][U]$ and, for each $q\in \mathbb{S}_{\lambda^+}$, $\dot{T}_q\in V[G][U_q]$ be, respectively, the $\mathbb{S}_{\lambda^+}$-name for the tree $T$  induced by $\tau$. Notice that the interpretation of the names $\tau$, $\dot{T}$ and $\dot{T}_q$ by the corresponding generic filters gives the same set; i.e. $T$. Thus, the only formal difference between these names is the ground model where they are regarded.
\begin{defi}
For a condition $p\in \mathbb{S}_{\lambda^+}$, write $\mathfrak{m}^p:=\max(\dom(g^p))$.
Denote by $\mathcal{S}$ the set of pairs $(g, H^*)$ for which there is $p\in\mathbb{S}_{\lambda^+}$ with $p_{\upharpoonright\mathfrak{m}^p+1}=(g,H^*)$ (c.f. Definition \ref{SinapovaForcing}).
We will consider $\mathcal{S}$ endowed with $\leq_{\mathcal{S}}$, the induced order by $\leq_{\mathbb{S}_{\lambda^+}}$: i.e. $(g,H^*)\leq_{\mathcal{S}} (i,I^*)$ iff there are $p,q\in\mathbb{S}_{\lambda^+}$ witnessings that $(g,H^*), (i,I^*)\in\mathcal{S}$ and $p\leq_{\mathbb{S}_{\lambda^+}} q$.

\end{defi}
The following property is implicitly considered in \cite{SinUncCof}.
\begin{defi}[Dagger property]\label{Dagger}
Work in $V[G][U]$.
For a pair $(g, H^*)\in\mathcal{S}$, we will say that $\dagger_{(g, H^*)}$ holds if there is $J\s \delta$ unbounded, $\langle p_\alpha\mid \alpha\in J\rangle$ a sequence of conditions in $\mathbb{S}_{\lambda^+}$ and $\xi<\kappa$  such that for each $\alpha\in J$ setting $u_\alpha:=\langle\alpha,\xi\rangle$, the following are true:
\begin{enumerate}
\item For each $\alpha\in J$, $p_\alpha$ witnesses that  $(g,H^*)\in\mathcal{S}$.
\item For each $\alpha<\beta$ in $J$, $p_\alpha\wedge p_\beta\Vdash^{V[G][U]}_{\mathbb{S}} u_\alpha<_{\dot{T}} u_\beta$.
\end{enumerate}
\end{defi}
Since $\mathbb{U}$ is $\delta$-directed closed (in $V$), $V[U]$ thinks that $\kappa$ is supercompact and  the same holds for the sequence $\langle\kappa_{\xi+1}\mid\xi<\mu\rangle$. 
 By appealing to the arguments of \cite[\S3]{SinUncCof} one has the following:
\begin{lemma}\label{DaggerSinapova}
In $V[G][U]$ the set $\{p\in\mathbb{S}_{\lambda^+}\mid \text{$\dagger_{p_{\upharpoonright\mathfrak{m}^p+1}}$ holds}\}$ is dense.
\end{lemma}

An immediate consequence of the previous lemma is the existence of a cofinal branch of $T$ in $V[\bar{R}]$ (see \cite[Proposition 21]{SinUncCof} and the subsequent discussion).
\begin{prop}
There is a cofinal branch $b\in V[\bar{R}]$ through $T$. 
\end{prop}

Now we are left to prove that $b$ induces a cofinal branch for $T$ in $V[R]$. Let $\dot{b}$ be a $\bar{\mathbb{R}}/G$-name for $b$. Moreover, let us assume that
$(\one_\mathbb{S}, \one_\mathbb{U})\Vdash_{\mathbb{S}\times\mathbb{U}}^{V[G]} \text{``$\dot{b}$ cofinal branch in $\tau$''}.$ We will need to consider a minor variation of the property $\dagger_h$ of \cite[Definition 3.3]{SinTree}.

\begin{notation}
\rm{
Work in $V[G]$. For a pair $(g, H^*)\in\mathcal{S}$, denote by $E_{(g,H^*)}$ the set of $u\in T$ for which there are $(q,r)\in \mathbb{S}\times \mathbb{U}$ such that $q$ witnesses $(g,H^*)\in\mathcal{S}$, $r\in U$ and $(q,r)\Vdash^{V[G]}_{\mathbb{S}\times\mathbb{U}} u\in \dot{b}$.}
\end{notation}

\begin{defi}\label{h-splitting}
Work in $V[G]$. For a pair $(g, H^*)\in\mathcal{S}$ and $\alpha<\delta$,  we say that there is a $(g, H^*)$-splitting at $u\in T_\alpha\cap E_{(g, H^*)}$ if, provided that $(q,r)$ witnesses $u\in E_{(g, H^*)}$, there are $\beta\geq  \alpha$, $v_1,v_2\in T_\beta$ and $r_1,r_2\leq_\mathbb{U} r$ in $U_q$ be such that 
\begin{itemize}
\item $(q,r_k)\Vdash_{\mathbb{S}\times \mathbb{U}}^{V[G]} v_k\in \dot{b}$, $k\in\{0,1\}$,
\item $q\Vdash_{\mathbb{S}}^{V[G][U]}v_1\perp_{\dot{T}} v_2$.
\end{itemize}
\end{defi}
\begin{remark}\label{RemarkOnSplittings}
\rm{
If there is a $(g,H^*)$-splitting at $u$ and $(g,I^*)\in\mathcal{S}$ then there is $(g,F^*)\leq_{\mathcal{S}} (g,I^*), (g,H^*)$ and a $(g,F^*)$-splitting at $u$. Indeed, let $q,r,v_1,v_2,r_1$ and $r_2$ witnessing the existence of a $(g,H^*)$-splitting at $u$. Now set $q^*:=(g,F)$, where 
$$F(\eta):=\begin{cases}
H^*(\eta)\cap I^*(\eta) , & \text{if $\eta\in\mathfrak{m}^p\setminus \dom(g^*)$};\\
H^q(\eta), & \text{$\mathfrak{m}^p<\eta$}.
\end{cases}
$$}
Set $F^*:=F\upharpoonright\mathfrak{m}^p+1$. Clearly $q^*\leq_{\mathbb{S}_{\lambda^+}} q$. By Remark \ref{RemarkInclusionGenerics}, $r_1,r_2\in U_{q^*}$. Evidently, $q^*,r,v_1,v_2,r_1$ and $r_2$ witness a $(g,F^*)$-splitting at $u$ and $(g,F^*)\leq_{\mathcal{S}} (g,I^*), (g,H^*)$. The same is true for $(g,F^*)=(g,I^*)$ if $(g,I^*)\leq_{\mathcal{S}} (g,H^*)$.
\end{remark}
This remark suggest the following definition:
\begin{defi}
Work in $V[G]$. For a stem $g$, we will say that there is a $g$-splitting at $u$ if there is some $(g,H^*)$-splitting at $u$, for some $(g,H^*)\in \mathcal{S}$.
\end{defi}


\begin{defi}
Work in $V[G][U]$. For a pair $(g, H^*)\in\mathcal{S}$ we will say that $\dagger^b_{(g, H^*)}$ holds if there is $J\s \delta$ unbounded, $\langle p_\alpha\mid \alpha\in J\rangle$ a sequence of conditions in $\mathbb{S}_{\lambda^+}$ and $\xi<\kappa$ such that for each $\alpha\in J$ setting $u_\alpha:=\langle\alpha,\xi\rangle$, the following are true:
\begin{enumerate}
\item For each $\alpha\in J$, $p_\alpha$ witnesses that  $(g,H^*)\in\mathcal{S}$.
\item For each $\alpha\in J$, $p_\alpha\Vdash^{V[G][U]}_{\mathbb{S}_{\lambda^+}} u_\alpha\in\dot{b}$.
\item For each $\alpha<\beta$ in $J$, $p_\alpha\wedge p_\beta\Vdash^{V[G][U]}_{\mathbb{S}_{\lambda^+}} u_\alpha<_{\dot{T}} u_\beta$.
\end{enumerate}
\end{defi}
 A straightforward modification of the arguments involved in the proof of Lemma \ref{DaggerSinapova} yields that $\{p\in\mathbb{S}_{\lambda^+}\mid \text{$\dagger^b_{p_{\upharpoonright\mathfrak{m}^p+1}}$ holds}\}$ is dense.
  
\begin{remark}\label{RemarkOnDagger}
\rm{ If $(g,I^*) \in\mathcal{S}$ and $\dagger^b_{(g,H^*)}$ holds then there is $(g,F^*)\leq_{\mathcal{S}} (g,I^*), (g,H^*)$ for which $\dagger^b_{(g,F^*)}$ holds. Indeed, let $J\s \delta$, $\langle p_\alpha\mid \alpha\in J\rangle$ and $\xi<\kappa$ witnessing $\dagger^b_{(g,H^*)}$. For each $\alpha\in J$, define $q_\alpha:=(g,F_\alpha)$, where $F_\alpha$ is defined as in Remark \ref{RemarkOnSplittings} but with respect to $H^{p_\alpha}\setminus \mathfrak{m}^{p_\alpha}+1$ rather than $H^p\setminus \mathfrak{m}^{p}+1$. It is obvious that $J$, $\langle q_\alpha\mid \alpha\in J\rangle$ and $\xi<\kappa$ are witness for $\dagger^b_{(g,F^*)}$. The same is true for $(g,F^*)=(g,I^*)$ if $(g,I^*)\leq_{\mathcal{S}} (g,H^*)$.
 }
 
\end{remark}
\begin{defi}
Work in $V[G][U]$. For a stem $g$, we will say that $\dagger^b_g$ holds if $\dagger^b_{(g,H^*)}$ holds, for some $(g,H^*)\in\mathcal{S}$. Define 
$$\alpha_{(g,H^*)}:=\sup \{\alpha<\delta\mid \exists u\in T_\alpha\cap E_{(g,H^*)}\,\text{and there is $(g,H^*)$-splitting at $u$}\},$$
and set $\alpha_g:=\sup\{\alpha_{(g,H^*)}\mid \exists H^*\,(g,H^*)\in\mathcal{S}\}$.
\end{defi}
By a very similar argument to Remark \ref{RemarkOnSplittings} if $(g,I^*)\leq_{\mathcal{S}} (g,H^*)$, then every $(g,H^*)$-splitting at some $u$ yields a $(g,I^*)$-splitting at $u$, and thus $\alpha_{(g,H^*)}\leq \alpha_{(g,I^*)}$.
\begin{lemma}\label{SplittingImpliesHsplitting}
If there is a $g$-splitting at $u$ then there is some stem $i\supseteq g$ for which there is a $i$-splitting at $u$ and $\dagger^b_i$ holds.
\end{lemma}
\begin{proof}
Let $u$ be some node where a $(g,H^*)$-splitting occurs, for some $H^*$. Say $(q,r)\Vdash^{V[G]}_{\mathbb{S}\times\mathbb{U}} u \in\dot{b}$, $(q,r_k)\Vdash^{V[G]}_{\mathbb{S}\times\mathbb{U}}v_k\in\dot{b}$, $r_k\leq_\mathbb{U} r$ and $r_k\in U_q$, for $k\in\{0,1\}$. By previous comments, find $\tilde{q}\leq_{\mathbb{S}_{\lambda^+}} q$ for which $\dagger^b_{\tilde{q}_{\upharpoonright \mathfrak{m}^{\tilde{q}}+1}}$ holds. Set $(i,I^*):=\tilde{q}_{\upharpoonright \mathfrak{m}^{\tilde{q}}+1}$. Hence, $\dagger^b_i$ holds.  By Remark \ref{RemarkInclusionGenerics}, $r_0,r_1\in U_{\tilde{q}}$. Clearly,  $\tilde{q}, r, v_1,v_2, r_1$ and $r_2$ witness the existence of a $(i,I^*)$-splitting at $u$.
\end{proof}
Now we need to show that if  $\dagger^b_{g}$ holds then  $\alpha_{g}<\delta$. This is essentially what is proved in \cite[Proposition 3.4]{SinTree} for Gitik-Sharon forcing. We will give some details just to convince the reader that the same arguments also work for Sinapova forcing.

\begin{lemma}\label{MainSplittingLemma}
In $V[G][U]$, for each stem $g$, if $\dagger^b_g$ holds then $\alpha_g<\delta$.
\end{lemma}
\begin{proof}
Assume otherwise and let $\bar{r}$ be a condition in $U$   such that $\bar{r}\Vdash^{V[G]}_{\mathbb{U}}\text{$``\dagger^b_{g}$ holds and $\dot{\alpha}_g=\check{\delta}$''}$. Since $\one_{\mathbb{U}}\Vdash^{V[G]}_{\mathbb{U}} \text{$``\delta$ is regular''}$ and $|\{H^*\mid (g,H^*)\in\mathcal{S}\}|^{V[G]}<\delta$, it follows that $$\bar{r}\Vdash^{V[G]}_{\mathbb{U}}\text{$``\exists \check{H}^*\,(\dagger^b_{(\check{g},\check{H}^*)}$ holds and $\dot{\alpha}_{(\check{g},\check{H}^*)}=\check{\delta})$''}.$$ By extending $\bar{r}$ if necessary, we may assume that there is $(g,H^*)\in\mathcal{S}$ be such that $\bar{r}\Vdash^{V[G]}_{\mathbb{U}}\text{$``\dagger^b_{(\check{g},\check{H}^*)}$ holds and $\dot{\alpha}_{(\check{g},\check{H}^*)}=\check{\delta}$''}$.
\begin{claim}\label{SplittingLemma}
Let $r\leq_{\mathbb{Q}}\bar{r}$ and $r\in U_q$, for some $q\in S$ witnessing $(g, I^*)\in\mathcal{S}$ and $(g,I^*)\leq_{\mathcal{S}} (g,H^*)$. Then in $V[G]$ there are  $\langle v^*_i\mid i<\varepsilon\rangle$ nodes and $\langle \langle p^*_i, r^*_i\rangle\mid i<\varepsilon\rangle$ conditions in $\mathbb{S}_{\lambda^+}\times\mathbb{U}$ be such that:
\begin{enumerate}
\item For each $i<\varepsilon$, $p^*_i\leq_{\mathbb{S}_{\lambda^+}} q$ , $r^*_i\leq_\mathbb{U} r$, $r_i\in U_{p^*_i}$;
\item for each $i<\varepsilon$, $p^*_i$ has stem $g$, 
\item for each $i<\varepsilon$, $\langle p^*_i,r^*_i\rangle\Vdash_{\mathbb{S}_{\lambda^+}\times \mathbb{U}} v^*_i\in \dot{b}$, and
\item for each $i<j<\varepsilon$, $p^*_i\wedge p^*_j\Vdash_{\mathbb{S}_{\lambda^+}} v^*_i\perp_{\dot{T}} v^*_j$.
\end{enumerate}
\end{claim}
\begin{proof}[Proof of claim]
Let $U'$ be a $\mathbb{U}/U_q$ generic over $V[G][U_q]$ and $r\in U'$. Since $r\leq_{\mathbb{Q}}\bar{r}$,  $\alpha_{(g,H^*)}=\delta$ and $\dagger^b_{(g,H^*)}$ hold in $V[G][U']$. By the previous remarks we have that $\dagger^b_{(g,I^*)}$ and $\alpha_{(g,I^*)}$ also hold in this model. Denote by $E_{(g,I^*)}$, $J$, $\langle p_\alpha\mid \alpha\in J\rangle$ and $\xi$ the objects in $V[G][U']$ that witness $\dagger^b_{(g,I^*)}$. Let us now work over $V[G][U']$. 

\begin{subclaim}\label{KeySubclaim}
For every $u\in E_{(g,I^*)}$, there is $p\in \mathbb{S}_{\lambda^+}$ with $p\leq^*_{\mathbb{S}_{\lambda^+}}q$, $r_1,r_2\in U_p$ and nodes $v_1,v_2$ of higher levels, such that $\langle p,r_k\rangle\Vdash_{\mathbb{S}_{\lambda^+}\times \mathbb{U}} v_k\in\dot{b}$ and $p\Vdash^{V[G][U']}_{\mathbb{S}_{\lambda^+}} v_1\perp_{\dot{T}} v_2, \, u<_{\dot{T}} v_1, \, u<_{\dot{T}} v_2.$
\end{subclaim}
\begin{proof}[Proof of subclaim]
Let $u\in E_{(g,I^*)}$ and $(p',t')\Vdash^{V[G]}_{\mathbb{S}_{\lambda^+}\times \mathbb{U}} u\in \dot{b}$ with $t'\in U$ and $p'$ witnessing $(g,I^*)\in \mathcal{S}$. Since $\alpha_{(g,I^*)}=\delta$, there is $v$ in a higher level of the tree for which there is a $(g,I^*)$-splitting. Namely, there are $p, r, v_1,v_2, r_1,r_2$ as follows:
\begin{enumerate}
\item $p\in\mathbb{S}_{\lambda^+}$ witnesses $(g,I^*)\in\mathcal{S}$, $r\in U$, $\langle p,r\rangle\Vdash^{V[G]}_{\mathbb{S}_{\lambda^+}\ast \mathbb{U}} v\in\dot{b}$,
\item $v_k$ is a node in a higher level than $v$ and $\langle p,r_k\rangle\Vdash^{V[G]}_{\mathbb{S}_{\lambda^+}\times \mathbb{U}}v_k\in\dot{b}$, with $r_k\leq_{\mathbb{U}} r$ and $r_k\in U_p$, for $k\in\{1,2\}$, 
\item $p\Vdash^{V[G][U']}_{\mathbb{S}_{\lambda^+}} v_1\perp_{\dot{T}} v_2$.
\end{enumerate}
Observe that we may further assume $r_1,r_2\leq_{\mathbb{U}} t'$. Also, $p^*:=p\wedge p'$ is a condition $\leq^*_{\mathbb{S}_{\lambda^+}}$-below $p$ and $p'$.  Remark \ref{RemarkInclusionGenerics} yields $r_1,r_2\in U_{p^*}$. Finally, notice that $p^*, r_1,r_2,v_1,v_2$ is a witness for our statement.
\end{proof}
By extending $r$ if necessary, we may assume that $r$ forces the conclusion of the above subclaim. Let $C$ be the set of all $\alpha<\delta$ such that for each $\beta<\alpha$ and $u\in T_\gamma$, if there is some $r'\leq_{\mathbb{U}/U_q} r$ with $r'\Vdash^{V[G][U_q]}_{\mathbb{U}/U_q} u\in \dot{E}_{(g,I^*)}$, then there are levels $\beta<\gamma_1\leq \gamma_2<\alpha$ and nodes $v_1\in T_{\gamma_1}$ and $v_2\in T_{\gamma_2}$  witnessing the above subclaim, for some conditions $p\in\mathbb{S}_{\lambda^+}$ and $r_1,r_2\in\mathbb{U}$. Clearly, $C$ is closed. Also, since $\alpha_{(g,I^*)}=\delta$, is unbounded, hence  $C$ is a club on $\delta$. Observe that $C\in V[G][U_q]$.

Working in $V[G][U']$ define $\langle p^i, \gamma_i,\alpha_i\mid i<\varepsilon\rangle$ as follows: $\gamma_i\in J$, $p^i:=p_{\gamma_i}$ and $\alpha_i\in C$ is such that $\gamma_i<\alpha_i\leq \gamma_{i+1}$. For each $i<\varepsilon$, set $u_i:=\langle \gamma_i, \xi\rangle$ and let $s_i\in U'$, $s_i\leq_{\mathbb{Q}} r$, be such that $s_i\Vdash^{V[G]}_{\mathbb{U}} \text{$``\gamma_i\in\dot{J}$ and $p^i=\dot{p}_{\gamma_i}$''}$. Since $\mathbb{A}$ is $\kappa$-cc and $\mathbb{U}$ is $\delta$-directed closed, Easton's lemma implies that $\mathbb{A}$ forces that $\mathbb{U}$ is $\delta$-distributive, hence $\langle p^i, \gamma_i,\alpha_i, s_i\mid i<\varepsilon\rangle\in V[G]$. By construction,
\begin{itemize}
\item for each $i<\varepsilon$, $p^i$ witnesses $(g,I^*)\in\mathcal{S}$,
\item for each $i<\varepsilon$, $\langle p^i,s_i\rangle\Vdash^{V[G]}_{\mathbb{S}_{\lambda^+}\times\mathbb{U}} u_i\in \dot{b}$,
\item $i<j<\varepsilon$, $p^i\wedge p^j\Vdash_{\mathbb{S}_{\lambda^+}} u_i<_{\dot{T}} u_j$.
\end{itemize}
In particular, $s_i\Vdash^{V[G]}_{\mathbb{U}} u_i\in \dot{E}_{(g,I^*)}$. By definition of $C$, for each $i<\varepsilon$, there is $q^i\leq^*_{\mathbb{S}_{\lambda^+}} q$,  $r^i_1,r^i_2\in U_{q^i}$ and $v^i_1$, $v^i_2$ be such that
\begin{enumerate}
\item for each $i<\varepsilon$ and $k\in\{1,2\}$, $\langle q^i, r^i_k\rangle\Vdash_{\mathbb{S}_{\lambda^+}\times\mathbb{U}} v^i_k\in\dot{b}$ and $r^i_k\in U_{q^i}$,
\item for each $i<\varepsilon$, $q_i\Vdash_{\mathbb{S}_{\lambda^+}} v^i_1\perp_{\dot{T}} v^i_2, \, u_i<_{\dot{T}} v^i_1, \, u_i<_{\dot{T}} v^i_2$,
\item for each $i<\varepsilon$, $\gamma_i<\mathrm{level}(v^i_1), \mathrm{level}(v^i_2)<\gamma_{i+1}$.
\end{enumerate}
Observe that we may further assume that $q^i\leq_{\mathbb{S}_{\lambda^+}} p^i$, as the stems are the same. Let $\varphi(i,k)$ be $\text{$``v^i_k<_{\dot{T}}u_{i+1}$''}$.  By (2) and the Prikry property, there is $k^*\in\{1,2\}$ and $p^*_i\leq^* q^i\wedge p^{i+1}$ be such that $p^*_i\Vdash_{\mathbb{S}_{\lambda^+}} \neg \varphi(i,k^*)$. Set 
 $r^*_i=r^i_{k^*}$ and $v^*_i:=v^i_{k^*}$. By using Remark \ref{RemarkInclusionGenerics} it is immediate that $\langle p^*_i, r^*_i, v^*_i\mid i<\varepsilon\rangle$ is as desired. This finishes the proof of the claim.
\end{proof}
From this point on the argument is identical to \cite{SinTree}, so we decline the chance to provide more details.

\end{proof}

\begin{lemma}
$V[R]\models \TP(\delta)$.
\end{lemma}
\begin{proof}[Proof skecth]
By Lemma \ref{MainSplittingLemma}, $\alpha^*:=\sup_{g}\{\alpha_{g}\mid \text{$``\dagger^b_g$ holds''} \}+1<\delta$. Let $u\in T_{\alpha^*}$ and $s^*\in R$ be such that $s^*\Vdash^{V[G]}_{\mathbb{R}^*/G} u\in\dot{b}$. Define $b^*:=\{v\in T\mid u<_T v,\, (\exists s\in R)\, s\leq_{\bar{\mathbb{R}}} s^*, s\Vdash_{\bar{\mathbb{R}}/G}^{V[G]} v\in\dot{b}\}$. Clearly, $b^*\in V[R]$ and $b^*$ is a cofinal set in $T$. By our initial assumption, $b^*$ is not a branch through $T$, hence there is $\gamma> \alpha^*$ with $|T_\gamma\cap b^*|\geq 2$. By Remark \ref{RemarkInclusionGenerics}, $\bar{\mathbb{R}}/R=\bigcup_{p\in S} R_p$. We can use this to prove that there is a $(g,H^*)$-splitting at $u$, for some $(g,H^*)$. Thus, $\alpha^*\leq \alpha_g$. By Lemma \ref{SplittingImpliesHsplitting}, we may further assume that $\dagger^b_{(g,H^*)}$ holds, so that $\alpha^*\leq \alpha_g< \alpha^*$. This forms the desired contradiction.
\end{proof}

\textbf{Acknowledgments:} The author would like to thank professor M. Golshani for suggesting him to work on this problem and to 
professor J. Bagaria for his friendly guidance. The said gratitude is also extensible to the anonymous referee for his/her carefully reading of the paper and for his/her timely corrections and remarks.

\bibliography{biblio}
\bibliographystyle{alpha}

\end{document}